\documentclass[a4paper,11pt]{amsart}
\usepackage[utf8x]{inputenc} 
\usepackage{graphicx}
\usepackage[bottom]{footmisc}
\usepackage{color}
\usepackage{float}
\usepackage{amsthm}
\usepackage{amssymb}
\usepackage{amsmath}
\usepackage{verbatim}
\usepackage{graphicx}
\usepackage{tikz}
\usepackage{tikz-cd}
\usepackage{stmaryrd} 

\usepackage{etoolbox}
\usepackage{dynkin-diagrams}
\def\row#1/#2!{#1_{\IfStrEq{#2}{}{n}{#2}} & \dynkin{#1}{#2}\\}

\usetikzlibrary{chains,arrows,positioning}
\usetikzlibrary{decorations.pathmorphing}
\usetikzlibrary{matrix,arrows}
\usetikzlibrary{shapes.geometric}

\usepackage[bbgreekl]{mathbbol}

\DeclareSymbolFontAlphabet{\mathbbm}{bbold}
\DeclareSymbolFontAlphabet{\mathbb}{AMSb}%

\numberwithin{equation}{section}

\theoremstyle{plain}
\newtheorem{theorem}{Theorem}[section]

\newtheorem{lemma}[theorem]{Lemma}

\theoremstyle{definition}
\newtheorem{defn}[theorem]{Definition}
\newtheorem{example}[theorem]{Example}

\theoremstyle{remark}
\newtheorem{remark}[theorem]{Remark}

\newtheorem*{note}{Note}

\newcommand{\inv}{^{-1}}

\newcommand{\Z}{\mathbb{Z}}

                                                                                                                                                                                                                                                                                                                                                                                                                                                                                                                                                                                                                                                                                                                                                                                                                                                                                                                                                                                                                                                                                                                                                                                                                                                                                                                                                                                                                                                                                                                                                                                                                                                                                                                                                                                                                                                                                                                                                                                                                                                                                                                                                                                                                                                                                                                                                                                                                                                                                                                                                                                                                                                                                                                                                                                                                                                                                                                                                                                                                                                                                                                                                                                                                                                                                                                                                                                                                                                                                                                                                                                                                                                                                                                                                                                                                                                                                                                                                                                                                                                                                                                                                                                                                                                                                                                                                                                                                                                                                                                                                                                                                                                                                                                                                                                                                                                                                                                                                                                                                                                                                                                                                                                                                                                                                                                                                                                                                                                                                                                                                                                                                                                                                                                                                                                                                                                                                                                                                                                                                                                                                                                                                                                                                                                                                                                                                                                                                                                                                                                                                                                                                                                                                                                                                                                                                                                                                                                                                                                                                                                                                                                                                                                                                                                                                                                                                                                                                                                                                                                                                                                                                                                                                                                                                                                                                                                                                                                                                                                                                                                                                                                                                                                                                                                                                                                                                                                                                                                                                                                                                                                                                                                                                                                                                                                                                                                                                                                                                                                                                                                                                                                                                                                                                                                                                                                                                                                                                                                                                                                                                                                                                                                                                                                                                                                                                                                                                                                                                                                                                                                                                                                                                                                                                                                                                                                                                                                                                                                                                                                                                                                                                                                                                                                                                                                                                                                                                                                                                                                                                                                                                                                                                                                                                                                                                                                                                                                                                                                                                                                                                                                                                                                                                                                                                                                                                                                                                                                                                                                                                                                                                                                                                                                                                                                                                                                                                                                                                                                                                                                                                                                                                                                                                                                                                                                                                                                                                                                                                                                                                                                                                                                                                                                                                                                                                                                                                                                                                                                                                                                                                                                                                                                                                                                                                                                                                                                                                                                                                                                                                                                                                                                                                                                                                                                                                                                                                                                                                                                                                                                                                                                                                                                                                                                                                                                                                                                                                                                                                                                                                                                                                                                                                                                                                                                                                                                                                                                                                                                                                                                                                                                                                                                                                                                                                                                                                                                                                                                                                                                                                                                                                                                                                                                                                                                                                                                                                                                                                                                                                                                                                                                                                                                                                                                                                                                                                                                                                                                                                                                                                          \newcommand{\tonnetz}{{\operatorname{tnz}}}                                                                                                                                                                                                                                                                                                                                                                                                                                                                                                                                                                                                                                                                                                                                                                                                                                                                                                                                                                                                                                                                                                                                                                                                                                                                                                                                                                                                                                                                                                                                                                                                                                                                                                                                                                                                                                                                                                                                                                                                                                                                                                                                                                                                                                                                                                                                                                                                                                                                                                                                                                                                                                                                                                                                                                                                                                                                                                                                                                                                                                                                                                                                                                                                                                                                                                                                                                                                                                                                                                                                                                                                                                                                                                                                                                                                                                                                                                                                                                                                                                                                                                                                                                                                                                                                                                                                                                                                                                                                                                                                                                                                                                                                                                                                                                                                                                                                                                                                                                                                                                                                                                                                                                                                                                                                                                                                                                                                                                                                                                                                                                                                                                                                                                                                                                                                                                                                                                                                                                                                                                                                                                                                                                                                                                                                                                                                                                                                                                                                                                                                                                                                                                                                                                                                                                                                                                                                                                                                                                                                                                                                                                                                                                                                                                                                                                                                                                                                                                                                                                                                                                                                                                                                                                                                                                                                                                                                                                                                                                                                                                                                                                                                                                                                                                                                                                                                                                                                                                                                                     
\newcommand{\Multisets}{{\operatorname{Multisets}}}
                                                                                                                                                                                                                                                                                                                                                                                                                                                                                                                                                                                                                                                                                                                                                                                                                                                                                                                                                                                                                                                                                                                                                                                                                                                                                                                                                                                                                                                                                                                                                                                                                                                                                                                                                                                                                                                                                                                                                                                                                                                                                                                                                                                                                                                                                                                                                                                                                                                                                                                                                                                                                                                                                                                                                                                                                                                                                                                                                                                                                                                                                                                                                                                                                                                                                                                                                                                                                                                                                                                                                                                                                                                                                                                                                                                                                                                                                                                                                                                                                                                                                                                                                                                                                                                              \newcommand{\Etonnetz}{{  {\operatorname{tnz}}}}

\tikzset{labl/.style={anchor=south, rotate=90, inner sep=.5mm}}

\title{Generalisations of Euler's Tonnetz on triangulated surfaces}
\author{Konstanze Rietsch}
\address{Department of Mathematics,
            King's College London,
            Strand, London
            WC2R 2LS
            UK
}
\email{konstanze.rietsch@kcl.ac.uk}%
\thanks{The author was supported by EPSRC grant EP/V002546/1}

\begin{document}

\begin{abstract}
We give a definition of a what we call a `tonnetz' on a triangulated surface, generalising the famous  tonnetz of Euler \cite{EulerTantamen}. In Euler's tonnetz the vertices of a regular `$A_2$ triangulation' of the plane are labelled with notes, or pitch-classes. In our generalisation we allow much more general labellings of triangulated surfaces. In particular, edge labellings turn out to lead to a rich set of examples. We construct natural examples that are related to crystallographic reflection groups and live on triangulations of tori. Underlying these we observe a curious relationship between mathematical Langlands duality and major/minor duality. We also construct `exotic' type-$A_2$ examples (different from Euler's Tonnetz), and a tonnetz on a sphere that encodes all major ninth chords. 
\end{abstract}

\maketitle
\section{Introduction}

Consider the regular triangulation of the plane by equilateral triangles. The famous tonnetz originating in the work of Euler \cite{EulerTantamen} can be thought of as placing a note, or pitch class, at every vertex of this triangulation, such that along the three directions notes go up in fifths, down in minor thirds, and down in major thirds, respectively, see Figure~\ref{f:EulerIntro}. This creates a beautiful pattern in which the major and minor triads appear alternatingly as triples of vertices of triangles. Much major work has been done studying the tonnetz in music theory, with references including \cite{Riemann, Oettingen, Lewin87, Cohn,Cohn98,TymoczkoBook1, Tymoczko12, CannasThesis}. We take a mathematical point of view as an approach to generalising the classical tonnetz, in particular using ideas from Lie Theory. For relevant mathematical background we refer to \cite{HillerBook,HumphreysBook,KirillovBook}. 

Let us consider the tonnetz with the lines in the triangulation viewed as reflection hyperplanes, and let us pick a triangle that we refer to as the `fundamental alcove'.  For Euler's tonnetz we may pick the $C-E-G$ triangle in the center of Figure~\ref{f:EulerIntro}, and consider $C$ as our origin. Then the group containing all the reflections along all the hyperplanes in this configuration is generated by reflections $s_0,s_1,s_2$ along the three lines bounding this fundamental alcove. The reflection called $s_0$ is along the line through the affine hyperplane with $E$ and $G$,
 $s_1$  is through the $x$-axis, and reflection $s_2$ is  through the line containing $C$ and $E$. We consider the reflections $s_1$ and $s_2$ to generate what is called the finite Weyl group of type $A_2$ (which is associated to the Lie algebra $\mathfrak {sl}_3$ and is just the symmetric group $S_3$, or symmetry group of an equilateral triangle). Adding in $s_0$ gives the simplest interesting example of an infinite Coxeter group, which is the so-called `affine Weyl group of type $A_2$'.  This point of view has already been part of group-theoretic investigations of harmonic movement, as in \cite{Schmidmeier}. Here, we will not so much go in the direction of group theory, or harmonic progressions. Rather, we will see how to generalise the concept of a tonnetz and the idea of geometric patterns of notes illuminating the structure of chords, with this point of view in mind.

We first generalise the notion of a tonnetz on a triangulated surface by allowing \textit{sets} of notes to be associated to vertices, edges and faces, with some compatibility conditions, as opposed to associating just single notes to vertices. Euler's tonnetz for example, associates a single note to every vertex, a pair to every edge, and a triad to every face. However, in our more general framework a tonnetz may also associate a single note to every edge, a triad to every triangle and a varying number of notes to every vertex, among other possibilities. The formal definition is given in Section~\ref{s:tonnetzdef}.

\begin{figure}
\centering
\begin{tikzpicture}[
  xscale=1.5,yscale=2.5981,
  note/.style={
  circle,  minimum size=0.5cm,fill=white},
  every node/.append style={font=\footnotesize}
  ]

\clip (0.2,0.25) rectangle (6.7,2.3);

\draw[pattern=north west lines, pattern color=lightgray] (3.5,1) -- (4,1.5) -- (4.5,1) -- cycle;



\draw (6.5,2)--(6.75,2.25);
\draw (6.5,2)--(6.25,2.25);
\draw (5.5,2)--(5.75,2.25);
\draw (5.5,2)--(5.25,2.25);
\draw (4.5,2)--(4.75,2.25);
\draw (4.5,2)--(4.25,2.25);
\draw (3.5,2)--(3.75,2.25);
\draw (3.5,2)--(3.25,2.25);
\draw (2.5,2)--(2.75,2.25);
\draw (2.5,2)--(2.25,2.25);
\draw (1.5,2)--(1.75,2.25);
\draw (1.5,2)--(1.25,2.25);
\draw (0.5,2)--(0.75,2.25);
\draw (0.5,2)--(0.25,2.25);

\draw (6.5,0)--(6.75,-0.25);
\draw (6.5,0)--(6.25,-0.25);
\draw (5.5,0)--(5.75,-0.25);
\draw (5.5,0)--(5.25,-0.25);
\draw (4.5,0)--(4.75,-0.25);
\draw (4.5,0)--(4.25,-0.25);
\draw (3.5,0)--(3.75,-0.25);
\draw (3.5,0)--(3.25,-0.25);
\draw (2.5,0)--(2.75,-0.25);
\draw (2.5,0)--(2.25,-0.25);
\draw (1.5,0)--(1.75,-0.25);
\draw (1.5,0)--(1.25,-0.25);
\draw (0.5,0)--(0.75,-0.25);
\draw (0.5,0)--(0.25,-0.25);

\begin{scope}
\newcommand*\columns{7}
\newcommand*\rows{2}
\clip(0,-\pgflinewidth) rectangle (\columns,\rows);
\foreach \y in {0,0.5,1,...,\rows} 
  \draw (0,\y) -- (\columns,\y);
\foreach \z in {-1.5,-0.5,...,\columns} 
{
    \draw (\z,\rows) -- (\z+2,0);
    \draw (\z,0) -- (\z+2,\rows);
}
\end{scope}

\foreach [count=\row from 0] \notelist in {
{E$\flat\flat$,B$\flat\flat$,F$\flat$,C$\flat$,G$\flat$,D$\flat$,A$\flat$},
  {C$\flat$,G$\flat$,D$\flat$,A$\flat$,E$\flat$,B$\flat$,F,C},
  {E$\flat$,B$\flat$,F,C,G,D,A},
  {C,G,D,A,E,B,F$\sharp$,C$\sharp$},
  {E,B,F$\sharp$,C$\sharp$,G$\sharp$,D$\sharp$,A$\sharp$}}
  \foreach \note [count=\column from 0,evaluate={\colX=\column+0.5-mod(\row,2)/2;}] in \notelist
     \node [note] at (\colX,\row*0.5) {\strut \note};
\end{tikzpicture}
\caption{Euler's tonnetz on an $A_2$ tiling of the plane, with fundamental alcove highlighted.}\label{f:EulerIntro}
\end{figure}
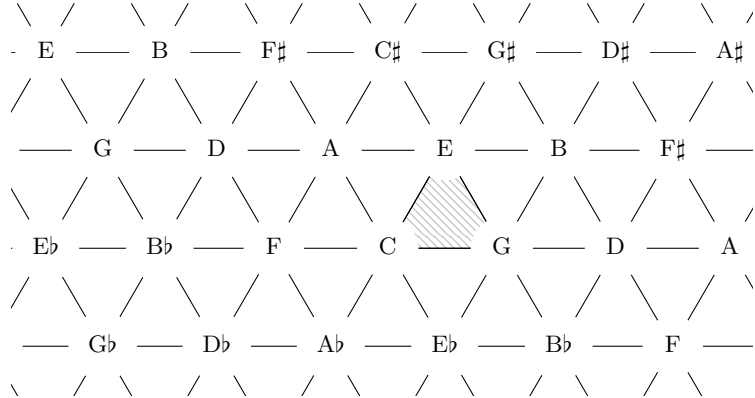
 
Next, we consider natural triangulations of the plane generalising the `type $A_2$ tiling' by equilateral triangles. These form the basis of some of our new tonnetz constructions. The words `type $A_2$' refer to the classification of simple Lie algebras into the classical series $A_n$, $B_n$, $C_n$, $D_n$ and exceptional types $E_6,E_7,E_8, F_4$ and $G_2$. These are encoded famously in their \textit{Dynkin diagrams}, the rank $2$ examples of which are
\begin{center} 
$A_2$\quad \dynkin[edge length=.8cm,
root radius=.09cm=]{A}{2}\qquad \qquad
$B_2$\quad \dynkin[edge length=.8cm,
root radius=.09cm]{B}{2} \qquad \qquad
$C_2$\quad \dynkin[edge length=.8cm,
root radius=.09cm]{C}{2}\qquad \qquad
$G_2$\quad \dynkin[edge length=.8cm,
root radius=.09cm]{G}{2}.
\end{center}
Considering the additional types $B_2,C_2,G_2$ now gives us further analogous hyperplane arrangements in the plane to utilize. We will refer to the tilings obtained from these hyperplane arrangements as $B_2$, $C_2$ and $G_2$ tilings. Here $B_2$ and $C_2$ represent the classical Lie algebras $\mathfrak{so}_5$ and $\mathfrak {sp}_4$, which are isomorphic, but we consider them as \textit{Langlands dual} (which they also are). Similarly, the Lie algebra $\mathfrak g_2$ is Langlands dual to itself, so that we secretly have another version of $G_2$ where the arrow in the Dynkin diagram points in the other direction.  

The vertices of the Dynkin diagrams may be considered as labelled by the simple reflections generating the finite Weyl group, such as
\begin{center} 
\dynkin[text style/.style={scale=1.1},edge length=1.2cm,
root radius=.09cm, labels={s_1,s_2.}, ]{A}{2}
\end{center} 
We describe these reflections in terms of the tonnetz by locating the associated side of the fundamental alcove and reading off its label. In this way we may label the Dynkin diagram instead with pitch classes. So, for Euler's tonnetz as in Figure~\ref{f:EulerIntro} with our choice of fundamental alcove we have a Dynkin diagram labelling
\begin{center} 
\dynkin[edge length=1.2cm,
root radius=.09cm, labels={C/G,C/E}]{A}{2}.
\end{center} 
Indeed, the Dynkin diagram can be read off from the fundamental alcove in rank $2$ as follows. The single, double or triple bond in the Dynkin diagram encodes the angle between the sides they label, which is $60^\circ$ in type $A_2$, $45^\circ$ in types $B_2/C_2$ and $30^\circ$ in type $G_2$. In the latter two cases the fundamental alcove is not equilateral, and the arrow in the Dynkin diagram points towards the longer side\footnote {This is the hyperplane perpendicular to the short root from the more general point of view of root systems.}. Swapping the direction of this arrow, and thus swapping which side is long and which is short, is part of a deep duality in mathematics called `Langlands duality'. 

After giving general axioms for what we call a \textit{simplicial surface tonnetz} in Section~\ref{s:tonnetzdef}, we start by focusing on crystallographic examples. In Section~\ref{s:BC} we construct an \textit{edge-tonnetz} (where every edge is labelled by a single pitch class) based on a $B_2$ tiling of the plane, and another `Langlands dual' one based on a $C_2$ tiling. In terms of Dynkin diagrams we consider these to give us labellings
\begin{center}
\dynkin[edge length=.8cm,
root radius=.09cm,labels={A,F}]{B}{2},\qquad
\quad \dynkin[edge length=.8cm,
root radius=.09cm,labels={A,F}]{C}{2}. 
\end{center}
In Section~\ref{s:G2} we construct an edge-tonnetz on a $G_2$-tiling and also its dual. These correspond to 
\begin{center}
\dynkin[edge length=.8cm,
root radius=.09cm,labels={A,D}]{G}{2} \qquad 
\quad \dynkin[edge length=.8cm,
root radius=.09cm,labels={D,A}]{G}{2}.
\end{center}
In all of these examples we see each triangle representing a major or minor triad, and the dualisation swaps the two types of triad. Each one of these examples also descends to give a finite tonnetz on a torus. 

In Section~\ref{s:tritone} we construct two versions of a tonnetz on the $A_2$ tiling. These have a pair of notes forming a tritone associated to each edge, and they realise the major triads based on a whole tone scale. The difference  between them relates to the allocation of pitch classes to the vertices, and has great impact on the symmetry group. Both of them descend to give a finite tonnetz on a torus but with very different fundamental domains.

In Section~\ref{s:sphere} we construct an edge tonnetz on a sphere. This tonnetz is made up of $24$ triangles paired up into $12$ quadrilaterals. Each quadrilateral has $5$ edges in total, having $4$ edges on the outside and one internal diagonal. Our tonnetz assigns a pitch class to each edge in a way as to realise precisely all $12$ major ninth chords, one in each quadrilateral.    
 
Finally, we remark that most of the constructions we make give edge tonnetz examples. With an edge tonnetz, unlike with Euler's vertex tonnetz, it is the case that adjacent chords can only be expected to have a single note in common. Our generalisation of the concept of a tonnetz in this paper is therefore not aimed at being specifically a tool for voice leading, for example. The most extreme case would be the final tonnetz on a sphere where, in particular, the $A$ and $B\flat$ major ninth chords are adjacent via a shared edge labelled $A$, their only in-common note. If we drop the `edge' constraint, adjacent chords may have no notes in common at all, see Section~\ref{s:tritone}. 

The constructions we make here serve to draw out and elucidate structures and patterns that exist and are inherently part of the framework of equal tempered music. Awareness of these relationships could hopefully also be a compositional tool.  
 \vskip .2cm 
 \noindent{\it Acknowledgements:} The author thanks Moreno Andreatta and Alexandre Popoff for stimulating conversations and encouragement early on. Thanks go to Pavel Tumarkin for a mathematical talk and Sara Dhillon for her inspiring jazz harmony lessons, which together led to the jazz bauble of major ninth chords. Deep thanks also go to Agnes Kory (and the BBCM).
 
\section{Axioms for a generalised tonnetz on a triangulated surface}\label{s:tonnetzdef}

Let $\mathcal N$ be the set of residue classes modulo $12$, so  $\mathcal N=\Z / 12 \Z$. We identify $\mathcal N$ with the pitch classes in the equal tempered chromatic scale, i.e. where notes related by an octave are considered equivalent, starting with $0\leftrightarrow A$. 
Let $\Multisets(\mathcal N)$ be the set of multisets in $\mathcal N$, where a multiset is a set where elements are allowed to have (finite) multiplicities. The order of a multiset $\mathcal C$ is denoted by $|\mathcal C|$, and the underlying set (without multiplicities) will be denoted by $[\mathcal C]$.  For example $\mathcal C=\{1,3,3\}$ is a multiset with $|\mathcal C|=3$ and $[\mathcal C]=\{1,3\}$. 
 Suppose $\mathcal A$ is a set and $\mathcal C$ is a multiset. By a bijection $\phi:\mathcal A\to\mathcal C$ we mean a map   $\phi:\mathcal A\to[\mathcal C]$ for which   the cardinality of any fiber $\phi\inv(c)$ recovers the multiplicity of $c$ in $\mathcal C$. For example, $\phi(1)= 1$, $\phi(2)=3$, $\phi(3)= 3$ defines a bijection  from $\mathcal A=\{1,2,3\}$ to $\mathcal C=\{1,3,3\}$.

Let $S$ be a surface (a $2$-dimensional toplogical manifold) and consider a triangulation of $S$ with sets of vertices, edges, and faces (also called $0$-simplices, $1$-simplices, and $2$-simplices) denoted $\mathcal V_0,\mathcal V_1$, and $\mathcal V_2$, respectively. For a more formal introduction to surfaces and triangulations we refer to \cite{GallierXu}. The set $\mathcal V=\bigcup_{i=0}^2 \mathcal V_i$ of all simplices in the triangulation comes with a partial order $\le$ given by inclusion. We write $\tau\prec \sigma$ for simplices $\tau$ and $\sigma$ if $\sigma$ covers $\tau$, that is, if $\tau $ has codimension $1$ in $\sigma$.

\begin{defn} \label{d:tonnetz}
We define a \textit{simplicial surface tonnetz} to be a triangulated surface $S$ together with a map from the set $\mathcal V=\mathcal V_0\cup\mathcal V_1\cup\mathcal V_2$ of simplices,
\[
\Etonnetz:\mathcal V\to \Multisets(\mathcal N),
\]
with the following properties.
\begin{enumerate}
\item (downwards coherent) 
Let  $\sigma\in \mathcal V_1\cup \mathcal V_2$. We have a bijection $\partial_\sigma:\{\tau\mid \tau\prec\sigma\}\to\Etonnetz(\sigma)$ such that $\partial_\sigma(\tau)=N$ implies $N\in \Etonnetz(\tau)$. 
\item (upwards coherent) Let $\rho\in \mathcal V_0\cup \mathcal V_1$. We have a bijection $\Delta_\rho:\{\tau\mid \tau\succ\rho\}\to \Etonnetz(\rho)$ such that $\Delta_\rho(\tau)=N$ implies  $N\in\Etonnetz (\tau)$.   
\end{enumerate} 
We note that any $1$-simplex has two vertices in its boundary, and is in the boundary of precisely two $2$-simplices. Therefore either form of coherence implies that $|\Etonnetz(\tau)|= 2$ for any $1$-simplex $\tau$.  
Downwards coherence also implies that $|\Etonnetz(\sigma)|= 3$ for any $2$-simplex $\sigma$. Upwards coherence implies that $|\Etonnetz(\rho)|$ recovers the valency for any vertex $\rho$. 
\end{defn}

We will write simply `tonnetz' instead of `simplicial surface tonnetz'. Note that the system of bijections is not part of the data of the tonnetz.
 
\begin{example}
For every $N\in\mathcal N$ there is a \textit{constant} tonnetz on $S$ for which $\tonnetz(\sigma)=\{N,N,N\}$ whenever $\sigma\in\mathcal V_2$.
\end{example}

We note that Definition~\ref{d:tonnetz} is very general, and there are additional properties that we will want to impose on a tonnetz in order to obtain interesting examples. For example, the following lemma shows that any set of triads strung together can arise in such a tonnetz. 
\begin{lemma}\label{l:toogen} Any map 
$\tonnetz_2:\mathcal V_2\to \Multisets(\mathcal N)$ with $|\tonnetz_2(\sigma)|=3$ extends to a tonnetz.\end{lemma}
 
\begin{proof} We can construct a tonnetz map $\tonnetz:\mathcal V\to\Multisets(\mathcal N)$  with $\tonnetz|_{\mathcal V_2}=\tonnetz_2$ as follows.
For every face $\sigma\in\mathcal V_2$ we define $\tonnetz(\sigma)=\tonnetz_2(\sigma)$, and we choose a bijection $\partial_\sigma:\{\tau\mid \tau\prec \sigma\}\to\tonnetz(\sigma)$. This bijection exists since both sides have cardinality $3$.

For $\tau\in\mathcal V_1$ with $\{\sigma\mid \sigma\succ \tau\}=\{\sigma_1,\sigma_2\}$ we define $\tonnetz(\tau)=\{\partial_{\sigma_1}(\tau),\partial_{\sigma_2}(\tau)\}$ and we  choose a bijection $\partial_\tau:\{\rho\mid \rho\prec \tau\}\to\tonnetz(\tau)$, which exists since both sides have cardinality $2$.

For a vertex $\rho\in\mathcal V_0$ with edges $\{\tau\mid \tau\succ \rho\}=\{\tau_1,\dots,\tau_k\}$ we finally define 
\[\tonnetz(\rho)=\{\partial_{\tau_1}(\rho),\dots,\partial_{\tau_k}(\rho)\}.
\]
The map $\tonnetz$  is now defined on all of $\mathcal V=\mathcal V_0\cup \mathcal V_1\cup \mathcal V_2$. The bijection $\Delta_\rho:\{\tau\mid \tau\succ \rho\}\to\tonnetz(\rho)$ can be defined by $\Delta_\rho(\tau_i):=\partial_{\tau_i}(\rho)$, and  $\Delta_\tau:\{\sigma\mid \sigma\succ \tau\}\to\tonnetz(\tau)$ by $\Delta_\tau(\sigma_i):=\partial_{\sigma_i}(\tau)$.
\end{proof}


\begin{defn}[Vertex and edge tonnetzes] Let $\tonnetz:\mathcal V\to\Multisets(\mathcal N)$ be a tonnetz on a surface $S$. If for all vertices $\rho\in\mathcal V_0$ it holds that the set $[\tonnetz(\rho)]$ underlying $\tonnetz(\rho)$ has cardinality one, then we call this tonnetz a \textit{vertex tonnetz}.  If for all edges $\tau\in\mathcal V_1$ it holds that the set $[\tonnetz(\tau)]$ underlying $\tonnetz(\tau)$ has cardinality one, then we call this tonnetz an \textit{edge tonnetz}. 

Note that an edge tonnetz is one where the dyads associated to the edges are all just two notes in unison. When describing an edge tonnetz we will generally suppress the multiset aspect and say a note $N$ is associated to a given edge, when we mean the multiset $\{N,N\}$ is associated to the edge. Similarly for a vertex tonnetz. 
If $\mathcal D$ is a set of dyad types, then we call a tonnetz a \textit{${\mathcal D}$-edge-tonnetz} whenever the dyads associated to the edges $\tau$ are all of the type described by $\mathcal D$. 
\end{defn}

\begin{defn}
We call a tonnetz \textit{major} if for every $2$-simplex $\sigma$ the set $\tonnetz(\sigma)$ represents a major triad, that is, is  of the form $\{N, N+4, N+7\}$ modulo $12$. Similarly, we call it \textit{minor, diminished} or \textit{augmented} if $\tonnetz(\sigma)$ is always a minor, diminished or augmented triad, respectively.
We call it a \textit{major/minor} tonnetz if $\tonnetz(\sigma)$ is either major or minor for all $2$-simplices $\sigma$. Finally, we call a major/minor tonnetz \textit{complete} if it contains every major and minor triad.
\end{defn}

\begin{lemma}\label{l:vertexedge} Let $\mathcal V=\mathcal V_0\cup\mathcal V_1\cup\mathcal V_2$ be the set of vertices, edges and faces of a triangulated surface~$S$. 
\begin{enumerate}
\item
Given any map $V:\mathcal V_0\to \mathcal N$, there is a unique vertex tonnetz such that for any vertex $\rho\in\mathcal V_0$ we have $[\tonnetz(\rho)]=\{V(\rho)\}$, and for every face $\sigma\in\mathcal V_2$ we have
\begin{equation*}\label{e:sigmadef}
\tonnetz(\sigma)=\{V(\rho)\mid\rho\in\mathcal V_0,\  \rho<\sigma \}.
\end{equation*}
\item 
Given any map $E:\mathcal V_1\to \mathcal N$, there is a unique edge tonnetz on $S$ with the property $[\tonnetz(\tau)]=\{E(\tau)\}$ for all $\tau\in\mathcal V_1$.
\end{enumerate}
\end{lemma}
\begin{proof}
(1) Let $\tau\in\mathcal V_1$ be an edge with vertices $\rho_1,\rho_2$. Downwards coherence and $[\tonnetz(\rho)]=\{V(\rho)\}$ completely determine the value of $\tonnetz(\tau)$. Namely, we must set
$\tonnetz(\tau)=\{V(\rho_1),V(\rho_2)\}$. Upwards coherence implies that to a vertex $\rho$ we must associate the multiset $\tonnetz(\rho)=\{V(\rho),\dotsc, V(\rho)\}$ with multiplicity of $V(\rho)$ given by the valency of $\rho$. For a face $\sigma$ the value of $\Etonnetz(\sigma)$ is already given.
It is straightforward to choose bijections that verify that altogether this defines a tonnetz. 

(2) To construct the edge tonnetz from the map $E$, first associate to every face $\sigma$ containing edges $\tau_1,\tau_2,\tau_3$ the multiset $\{E(\tau_1),E(\tau_2),E(\tau_3)\}$. For a vertex $\rho$ which is contained in edges $\tau^1,\dotsc,\tau^k$ we set $\tonnetz(\rho)=\{E(\tau^1),E(\tau^2),\dotsc,E(\tau^k)\}$.   Again the coherence conditions are straightforward, and they force the uniqueness of this tonnetz. 
\end{proof}

\begin{example} Lemma~\ref{l:vertexedge}(1) gives us a standard way to construct a tonnetz from any vertex labelling on a triangulated surface. 
Euler's tonnetz \cite{EulerTantamen} can be interpreted as an example of this. It is a complete major/minor vertex tonnetz on an $A_2$-tiling of the plane in our terminology. Or, we may take the quotient of the tonnetz using its natural $\Z^2$ translational symmetry (for which a choice of fundamental domain is indicated in Figure~\ref{f:Euler}), so that we obtain a complete major/minor vertex tonnetz on a torus. Other vertex tonnetzes, variations on Euler's tonnetz where the intervals have been adjusted, were constructed in~\cite{Cohn}.  
With our very general definition, even keeping the same underlying tiling, we can change the tonnetz much more radically. One way to do this is demonstrated in Section~\ref{s:tritone}.
\end{example}

\begin{remark}
Vertex tonnetzes as in the above example are the more common approach to the organisation of musical structure via geometry. Interesting surfaces other than tori have appeared in this context in the literature. For example, the idea of a vertex tonnetz on a Klein bottle was introduced in \cite{Peck}. Replacing triangles in our definition with more general polygons gives natural further generalisation of our Definition~\ref{d:tonnetz}. In the case of pentagons, vertex tonnetzes of this more general type were constructed on the hyperbolic plane (and thereby on higher genus surfaces) in \cite{Piovan}. A generalisation of Euler's tonnetz similiar to Cohn \cite{Cohn} but from $12$ to $N$-tone equal tempered scales, with classification of surfaces arising, was undertaken in \cite{Catanzaro-tonnetze}.
\end{remark}

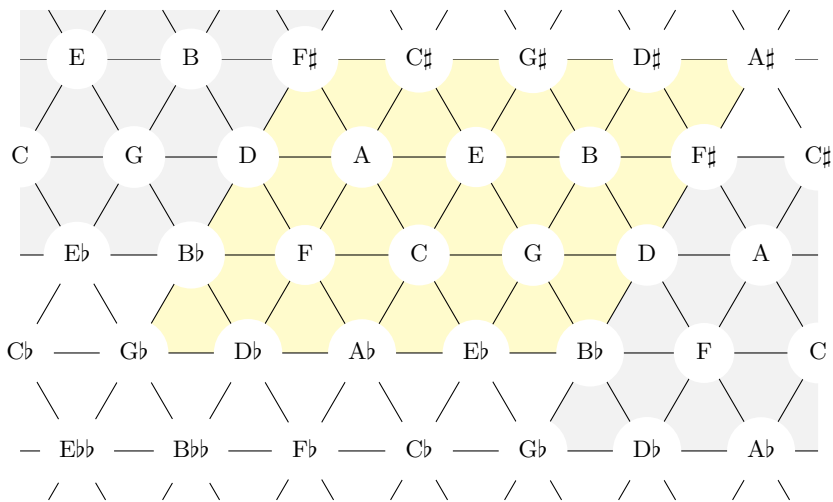
\begin{figure}
\centering
\begin{tikzpicture}[
  xscale=1.5,yscale=2.5981,
  note/.style={
  circle,  minimum size=0.5cm,fill=white},
  every node/.append style={font=\footnotesize}
  ]

\fill[yellow!25] (1,0.5) -- (2.5,2) -- (6.5,2) -- (5,0.5) -- cycle;

\fill[lightgray!20] (4.5,0) -- (6,1.5) -- (7,1.5) -- (7,0) -- cycle;

\fill[lightgray!20] (1.5,1) -- (2.75,2.25) -- (0,2.25) -- (0,1) -- cycle;

\draw (6.5,2)--(6.75,2.25);
\draw (6.5,2)--(6.25,2.25);
\draw (5.5,2)--(5.75,2.25);
\draw (5.5,2)--(5.25,2.25);
\draw (4.5,2)--(4.75,2.25);
\draw (4.5,2)--(4.25,2.25);
\draw (3.5,2)--(3.75,2.25);
\draw (3.5,2)--(3.25,2.25);
\draw (2.5,2)--(2.75,2.25);
\draw (2.5,2)--(2.25,2.25);
\draw (1.5,2)--(1.75,2.25);
\draw (1.5,2)--(1.25,2.25);
\draw (0.5,2)--(0.75,2.25);
\draw (0.5,2)--(0.25,2.25);

\draw (6.5,0)--(6.75,-0.25);
\draw (6.5,0)--(6.25,-0.25);
\draw (5.5,0)--(5.75,-0.25);
\draw (5.5,0)--(5.25,-0.25);
\draw (4.5,0)--(4.75,-0.25);
\draw (4.5,0)--(4.25,-0.25);
\draw (3.5,0)--(3.75,-0.25);
\draw (3.5,0)--(3.25,-0.25);
\draw (2.5,0)--(2.75,-0.25);
\draw (2.5,0)--(2.25,-0.25);
\draw (1.5,0)--(1.75,-0.25);
\draw (1.5,0)--(1.25,-0.25);
\draw (0.5,0)--(0.75,-0.25);
\draw (0.5,0)--(0.25,-0.25);

\begin{scope}
\newcommand*\columns{7}
\newcommand*\rows{2}
\clip(0,-\pgflinewidth) rectangle (\columns,\rows);
\foreach \y in {0,0.5,1,...,\rows} 
  \draw (0,\y) -- (\columns,\y);
\foreach \z in {-1.5,-0.5,...,\columns} 
{
    \draw (\z,\rows) -- (\z+2,0);
    \draw (\z,0) -- (\z+2,\rows);
}
\end{scope}

\foreach [count=\row from 0] \notelist in {
{E$\flat\flat$,B$\flat\flat$,F$\flat$,C$\flat$,G$\flat$,D$\flat$,A$\flat$},
  {C$\flat$,G$\flat$,D$\flat$,A$\flat$,E$\flat$,B$\flat$,F,C},
  {E$\flat$,B$\flat$,F,C,G,D,A},
  {C,G,D,A,E,B,F$\sharp$,C$\sharp$},
  {E,B,F$\sharp$,C$\sharp$,G$\sharp$,D$\sharp$,A$\sharp$}}
  \foreach \note [count=\column from 0,evaluate={\colX=\column+0.5-mod(\row,2)/2;}] in \notelist
     \node [note] at (\colX,\row*0.5) {\strut \note};
\end{tikzpicture}
\caption{Euler's tonnetz is a vertex tonnetz in the plane that descends to a tonnetz on a torus.}\label{f:Euler}
\end{figure}

\section{Edge tonnetzes of type $B_2$ and $C_2$}\label{s:BC}

In this section we describe one tonnetz for a tiling of type $B_2$ and one for $C_2$. The first one is entirely major and the second entirely minor. 

\subsection{}\label{s:B2} The tonnetz based on the  $B_2$-tiling, has four distinct triads, all major. It is given in Figure~\ref{f:B2}, and consists of an infinite tiling of the plane with fundamental domain highlighted in yellow. Gluing together the edges of the fundamental domain we obtain a finite tonnetz on a torus. It has $8$ triangles but only $4$ distinct triads, each one appearing twice.  Namely these triads are the major triads on $F, A\flat, B$ and~$D$.

There are two types of vertices in this tonnetz. One of them has $8$  edges, and correspondingly $8$ notes, attached to it, namely
\[
\begin{array}{cccc}
A, & C,& E\flat, & G\flat,\\
 B,& D &F&A\flat.\\
\end{array}
\]
On the torus there are two such $8$-valent vertices. The other  type of vertex is $4$-valent and has the diminished seventh chord $A,  C, E\flat,  G\flat$  associated to it. There are also two of those.

If we transpose every note up by a whole tone or down by a whole tone we get two alternative versions of this tonnetz, which altogether cover all of the $12$ major triads. If we now label the vertices of the Dynkin diagram by its corresponding edges for the fundamental alcove, then the three major edge tonnetzes described in this section correspond to,

\begin{center}
$B_2$\quad \dynkin[edge length=.8cm,
root radius=.09cm,labels={A,F}]{B}{2} (pictured),\qquad
\quad \dynkin[edge length=.8cm,
root radius=.09cm,labels={B,G}]{B}{2},\qquad 
\quad \dynkin[edge length=.8cm,
root radius=.09cm,labels={G,E\flat}]{B}{2}.
\end{center}

 \begin{figure}   
\centering
\begin{tikzpicture}[baseline= (a).base]
\fill[yellow!25] (-1.95,1.7) -- (1.95,1.7) -- (1.95,-1.65) -- (-1.95,-1.65)-- cycle;
\fill[ pattern=north west lines, pattern color=lightgray] (1.95,1.7) -- (0,0) -- (1.95,0)-- cycle;
\node[scale=.91,xscale=.9] (a) at (0,0){
\begin{tikzcd}
	\bullet && \bullet && \bullet && \bullet && \bullet && \bullet && \bullet \\
	\\
	\bullet && \bullet && \bullet && \bullet && \bullet && \bullet && \bullet \\
	\\
	\bullet && \bullet && \bullet && \bullet && \bullet && \bullet && \bullet
	\arrow["E\flat"{description}, no head, from=1-5, to=1-7]
	\arrow["A"{description}, no head, from=1-7, to=1-9]
	\arrow["D\sharp"{description}, no head, from=5-5, to=5-7]
	\arrow["A"{description}, no head, from=5-7, to=5-9]
	\arrow["F\sharp"{description}, no head, from=5-9, to=3-9]
	\arrow["C"{description}, no head, from=3-9, to=1-9]
	\arrow["A"{description}, no head, from=3-7, to=3-9]
	\arrow["C"{description}, no head,
	from=3-7, to=1-7]
	\arrow["F\sharp"{description}, no head, from=3-7, to=5-7]
	\arrow["D"{description}, no head, 
	from=3-7, to=5-9]
	\arrow["B"{description}, no head, from=3-7, to=5-5]
	\arrow["F"{description}, no head, from=3-7, to=1-9]
	\arrow["A\flat"{description}, no head, from=3-7, to=1-5]
	\arrow["E\flat"{description}, no head, from=1-9, to=1-11]
	\arrow["E\flat"{description}, no head, from=3-9, to=3-11]
	\arrow["C"{description}, no head, from=1-11, to=3-11]
	\arrow["E\flat"{description}, no head, from=3-5, to=3-7]
	\arrow["C"{description}, no head, from=1-5, to=3-5]
	\arrow["F\sharp"{description}, no head, from=5-5, to=3-5]
	\arrow["A"{description}, no head, from=3-11, to=3-13]
	\arrow["A"{description}, no head, from=1-11, to=1-13]
	\arrow["C"{description}, no head, from=1-13, to=3-13]
	\arrow["F\sharp"{description}, no head, from=3-13, to=5-13]
	\arrow["A"{description}, no head, from=5-9, to=5-11]
	\arrow["F\sharp"{description}, no head, from=3-11, to=5-11]
	\arrow["D\sharp"{description}, no head, from=5-11, to=5-13]
	\arrow["A"{description}, no head, from=3-5, to=3-3]
	\arrow["C"{description}, no head, from=3-3, to=1-3]
	\arrow["A"{description}, no head, from=1-3, to=1-5]
	\arrow["F"{description}, no head, from=3-11, to=1-13]
	\arrow["A\flat"{description}, no head, from=1-9, to=3-11]
	\arrow["D"{description}, no head, from=3-11, to=5-13]
	\arrow["B"{description}, no head, from=5-9, to=3-11]
	\arrow["F\sharp"{description}, no head, from=3-3, to=5-3]
	\arrow["A"{description}, no head, from=5-3, to=5-5]
	\arrow["D"{description}, no head, from=3-3, to=5-5]
	\arrow["F"{description}, no head, from=3-3, to=1-5]
	\arrow["E\flat"{description}, no head, from=1-3, to=1-1]
	\arrow["C"{description}, no head, from=1-1, to=3-1]
	\arrow["D\sharp"{description}, no head, from=5-1, to=5-3]
	\arrow["E\flat"{description}, no head, from=3-1, to=3-3]
	\arrow["B"{description}, no head, from=3-3, to=5-1]
	\arrow["A\flat"{description}, no head, from=1-1, to=3-3]
	\arrow["F\sharp"{description}, no head, from=3-1, to=5-1]
\end{tikzcd}};
\end{tikzpicture}
\caption{A major edge-tonnetz on a $B_2$ tiling of the plane with fundamental alcove shown in gray. This tonnetz tiles the plane doubly periodically. Alternatively it quotients to a major edge-tonnetz on a torus with $8$ triangles, $12$ edges, and $4$ vertices. }\label{f:B2}
\end{figure}
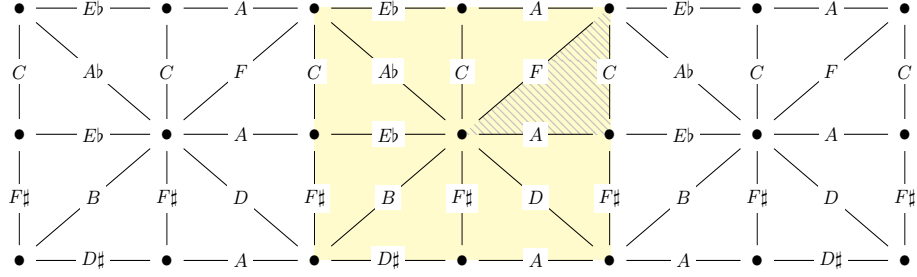

\begin{remark}\label{r:angles} There are many $90^\circ$ angles in this $B_2$ tonnetz, which is what leads to the diminished seventh chords appearing here. Namely, a rotation by $90^\circ$ degrees is a symmetry of order $4$, as is transposition by a minor third. 
Where the $90^\circ$ angles are split into two $45^\circ$ angles we have split the minor third (modulo the octave) into a minor sixth and a perfect fifth. If we didn't care about using the $12$-note system, we could have chosen instead of $F$ the note half-way between $E$ and $F$ for the long edge of the fundamental alcove. This would have given us a completely regular arrangement with a $45^\circ$ symmetry about the center.  
\end{remark}
\subsection{}\label{s:C2}
To obtain a $C_2$ version of the $B_2$ tonnetz from Section~\ref{s:B2}, we need to reverse the arrow in the Dynkin diagram, which means we swap which edge is short and which is long in the fundamental alcove. We still label the first edge of the fundamental alcove (along the $x$-axis) by $A$, as before, and we label the second again by $F$. But now we have a right angle at the end of the edge labelled $F$, meaning the remaining  side of the fundamental alcove should be a minor third away from $F$, see Remark~\ref{r:angles}. This leads us to choose $D$ for this side, and so the triad associated to the fundamental alcove is $D,F,A$, that is, $D$-minor.

The resulting $C_2$ tonnetz is shown in Figure~\ref{f:C2}. It is analogous to the $B_2$ one, but with major triads  replaced by minor ones. We can again transpose up or down a tone, and obtain in total three minor edge-tonnetzes, such that every minor triad is found in precisely one of them,

\begin{center}
$C_2$\quad \dynkin[edge length=.8cm,
root radius=.09cm,labels={A,F}]{C}{2} (pictured),\qquad
\quad \dynkin[edge length=.8cm,
root radius=.09cm,labels={B,G}]{C}{2},\qquad 
\quad \dynkin[edge length=.8cm,
root radius=.09cm,labels={G,E\flat}]{C}{2} .
\end{center}

Finally, of course there is a choice underlying which gives major and which gives minor chords, between $B_2$ and $C_2$. For example, replacing $F$ by $E$ in the Dynkin diagrams associated to the pictured tonnetz diagrams switches these roles. 
\begin{remark}
We note that for the type $B_2$ tonnetz we constructed and its type $C_2$ counterpart, while they are different, one containing the $F,A\flat, B,D$ major chords, and the other their minor versions, they are `the same' in terms of key signatures associated to the chords that appear. For example, $D$-minor is the relative minor for $F$ major and so forth. The two constructions also have in common the absence of the same four notes, $A\sharp, C\sharp, E,G$, which are exactly the notes that 
do not appear in any of the major or minor triads of $F,A\flat, B$ and $D$.
\end{remark}
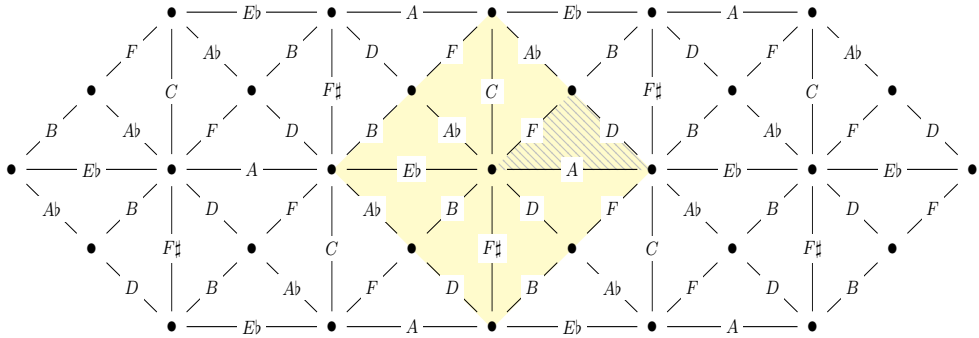
\begin{figure}
\begin{tikzpicture}[baseline= (a).base]

\fill[yellow!25] (-2.1,0) -- (0,2.1) -- (2.1,0) -- (0,-2.1)-- cycle;

\fill[ pattern=north west lines, pattern color=lightgray]  (0,0) -- (2.1,0) -- (1.05,1.03)-- cycle;

\node[scale=.91,xscale=.8] (a) at (0,0){
\begin{tikzcd}
	&& \bullet && \bullet && \bullet && \bullet && \bullet & {} \\
	& \bullet && \bullet && \bullet && \bullet && \bullet && \bullet \\
	\bullet && \bullet && \bullet && \bullet && \bullet && \bullet && \bullet \\
	& \bullet && \bullet && \bullet && \bullet && \bullet && \bullet \\
	&& \bullet && \bullet && \bullet && \bullet && \bullet
	\arrow["A"{description}, no head, from=3-7, to=3-9]
	\arrow["F\sharp"{description}, no head, from=3-7, to=5-7]
	\arrow["E\flat"{description}, no head, from=3-7, to=3-5]
	\arrow["B"{description}, no head, from=3-5, to=2-6]
	\arrow["F"{description}, no head, from=2-6, to=1-7]
	\arrow["C"{description}, no head, from=3-7, to=1-7]
	\arrow["A\flat"{description}, no head, from=3-7, to=2-6]
	\arrow["A\flat"{description}, no head, from=2-8, to=1-7]
	\arrow["D"{description}, no head, from=2-8, to=3-9]
	\arrow["F"{description}, no head, from=3-7, to=2-8]
	\arrow["B"{description}, no head, from=5-7, to=4-8]
	\arrow["F"{description}, no head, from=4-8, to=3-9]
	\arrow["D"{description}, no head, from=3-7, to=4-8]
	\arrow["A\flat"{description}, no head, from=3-5, to=4-6]
	\arrow["D"{description}, no head, from=4-6, to=5-7]
	\arrow["B"{description}, no head, from=3-7, to=4-6]
	\arrow["B"{description}, no head, from=2-8, to=1-9]
	\arrow["A\flat"{description}, no head, from=4-8, to=5-9]
	\arrow["E\flat"{description}, no head, from=5-9, to=5-7]
	\arrow["E\flat"{description}, no head, from=1-7, to=1-9]
	\arrow["F\sharp"{description}, no head, from=1-9, to=3-9]
	\arrow["A"{description}, no head, from=1-9, to=1-11]
	\arrow["D"{description}, no head, from=1-9, to=2-10]
	\arrow["B"{description}, no head, from=3-9, to=2-10]
	\arrow["F"{description}, no head, from=2-10, to=1-11]
	\arrow["A\flat"{description}, no head, from=2-10, to=3-11]
	\arrow["C"{description}, no head, from=1-11, to=3-11]
	\arrow["E\flat"{description}, no head, from=3-9, to=3-11]
	\arrow["C"{description}, no head, from=3-9, to=5-9]
	\arrow["F"{description}, no head, from=5-9, to=4-10]
	\arrow["A\flat"{description}, no head, from=3-9, to=4-10]
	\arrow["B"{description}, no head, from=4-10, to=3-11]
	\arrow["D"{description}, no head, from=4-10, to=5-11]
	\arrow["A"{description}, no head, from=5-9, to=5-11]
	\arrow["F\sharp"{description}, no head, from=5-11, to=3-11]
	\arrow["A\flat"{description}, no head, from=1-11, to=2-12]
	\arrow["F"{description}, no head, from=3-11, to=2-12]
	\arrow["D"{description}, no head, from=2-12, to=3-13]
	\arrow["E\flat"{description}, no head, from=3-11, to=3-13]
	\arrow["D"{description}, no head, from=3-11, to=4-12]
	\arrow["B"{description}, no head, from=5-11, to=4-12]
	\arrow["F"{description}, no head, from=4-12, to=3-13]
	\arrow["D"{description}, no head, from=2-6, to=1-5]
	\arrow["A"{description}, no head, from=1-5, to=1-7]
	\arrow["B"{description}, no head, from=1-5, to=2-4]
	\arrow["D"{description}, no head, from=2-4, to=3-5]
	\arrow["F"{description}, no head, from=2-4, to=3-3]
	\arrow["D"{description}, no head, from=3-3, to=4-4]
	\arrow["F"{description}, no head, from=4-4, to=3-5]
	\arrow["A\flat"{description}, no head, from=4-4, to=5-5]
	\arrow["F"{description}, no head, from=5-5, to=4-6]
	\arrow["A"{description}, no head, from=5-7, to=5-5]
	\arrow["B"{description}, no head, from=4-4, to=5-3]
	\arrow["E\flat"{description}, no head, from=5-3, to=5-5]
	\arrow["D"{description}, no head, from=5-3, to=4-2]
	\arrow["B"{description}, no head, from=4-2, to=3-3]
	\arrow["E\flat"{description}, no head, from=1-5, to=1-3]
	\arrow["A\flat"{description}, no head, from=1-3, to=2-4]
	\arrow["F"{description}, no head, from=1-3, to=2-2]
	\arrow["A\flat"{description}, no head, from=2-2, to=3-3]
	\arrow["B"{description}, no head, from=2-2, to=3-1]
	\arrow["A\flat"{description}, no head, from=3-1, to=4-2]
	\arrow["C"{description}, no head, from=1-3, to=3-3]
	\arrow["F\sharp"{description}, no head, from=1-5, to=3-5]
	\arrow["F\sharp"{description}, no head, from=3-3, to=5-3]
	\arrow["C"{description}, no head, from=3-5, to=5-5]
	\arrow["A"{description}, no head, from=3-3, to=3-5]
	\arrow["E\flat"{description}, no head, from=3-3, to=3-1]
\end{tikzcd}};
\end{tikzpicture}
\caption{A minor edge-tonnetz on a type $C_2$ tiling of the plane with fundamental alcove indicated. This tonnetz quotients to a minor edge-tonnetz on a torus with chords, $D$-minor, $F$-minor, $A\flat$-minor and $B$-minor, each appearing twice.}\label{f:C2}
\end{figure}

\section{A dual pair of major/minor edge tonnetzes of type $G_2$}\label{s:G2}
 
The $G_2$ tiling of the plane hosts a very natural major/minor edge-tonnetz. Namely, this tonnetz is based on the circle of fifths, and features six major and six minor triads in an alternating pattern. Langlands duality transforms this tonnetz into its complementary one which contains precisely the remaining triads, turning ones that were originally major into minor ones and vice versa.

In Figure~\ref{f:G2-full}, the first $G_2$-tonnetz can be seen. We have a fundamental domain that periodically tiles the plane, namely in this case it is a hexagon comprising $12$ triangles. Its edge labelling is completely shown on the right-hand side of Figure~\ref{f:G2-full}. In the middle it is indicated what are the major and minor (lower-case) triads appearing. On the left-hand side the outer edge labels are colour-coded to indicate how to glue to a torus. The actual gluing is illustrated in Figure~\ref{f:G2ontorus}.
\begin{figure}
\begin{center}
\begin{tikzpicture}[]
\clip (-7,-2.8) rectangle (7.5,2.7);
\node[rotate=30,scale=.91] (a) at (0,0){
\begin{tikzpicture}[scale=0.8,
  xscale=3,yscale=5.199, 
  note/.style={
  minimum size=0.75cm},
  every node/.append style={font=\footnotesize}
  ]

\draw[fill=yellow!25] (2,0.5) -- (1.5,1) -- (2,1.5) -- (3,1.5) -- (3.5,1)--(3,0.5)-- cycle;

\fill[ pattern=north west lines, pattern color=lightgray]  (3.25,1.25) -- (3,1.5) -- (2.5,1)-- cycle;




\begin{scope}
\newcommand*\columns{5}
\newcommand*\rows{2}
\clip(0,-\pgflinewidth) rectangle (\columns,\rows);
\foreach \y in {0,0.5,1,...,\rows} 
  \draw (0,\y) -- (\columns,\y);
\foreach \z in {-1.5,-0.5,...,\columns} 
{
    \draw (\z,\rows) -- (\z+2,0);
    \draw (\z,0) -- (\z+2,\rows);
}
\end{scope}

\draw (4,1.7)--(4,-0.25);
\draw (2.5,2.3)--(2.5,-0.25);
\draw (1,2.15)--(1,0.25);
\draw (0.7,0.4)--(4.3,1.6);
\draw (0,1.17)--(2.95,2.15);
\draw (2.15,-0.12)--(5,0.837);
\draw (0,1.83)-- (5,0.17);

\draw (2.5,2)--(2.65,2.15);
\draw (2.5,2)--(2.35,2.15);
\draw (1.5,2)--(1.65,2.15);
\draw (1.5,2)--(1.35,2.15);
\draw (0.5,2)--(0.65,2.15);
\draw (0.5,2)--(0.35,2.15);

\draw(4.5,0)--(4.75,-0.25);
\draw (4.5,0)--(4.25,-0.25);
\draw (3.5,0)--(3.75,-0.25);
\draw (3.5,0)--(3.25,-0.25);
\draw[red] (2.5,0)--(2.75,-0.25);
\draw[red] (2.5,0)--(2.25,-0.25);
\draw (1.5,0)--(1.75,-0.25);
\draw (1.5,0)--(1.25,-0.25);
\draw (0.5,0)--(0.75,-0.25);
\draw (0.5,0)--(0.25,-0.25);


\draw (0.7,2) node [rotate=-30,rectangle,fill=red!25] {$E$};
\draw (1.2,2) node [rotate=-30,rectangle,fill=red!25] {$B\flat$};

\draw (0.37,1.87) node [rotate=-30,rectangle,fill=purple!15] {$A\flat$};
\draw (1.64,1.84) node [rotate=-30,rectangle,fill=green!20] {$F\sharp$};

\draw (0.1,1.62) node [rotate=-30,rectangle,fill=purple!15] {$D$};
\draw (1.88,1.63) node[rotate=-30,rectangle,fill=green!20]  {$C$};

\draw (0.12,1.37) node [rotate=-30,rectangle,fill=green!20] {$G\flat$};
\draw (1.75,1.37) node [rotate=-30,rectangle,fill=purple!15] {$G\sharp$};

\draw (0.37,1.13) node [rotate=-30,rectangle,fill=green!20] {$C$};
\draw (1.55,1.15) node   [rotate=-30,rectangle,fill=purple!15] {$D$};

\draw (0.7,1) node [rotate=-30,rectangle,fill=red!25] {$E$};
\draw (1.2,1) node [rotate=-30,rectangle,fill=red!25] {$A\sharp$};

]


\draw (3,1.27) node[rotate=-30,circle,draw] {$\mathbf D$};

\draw (3.2,1.11) node[rotate=-30,circle,draw] {$\mathbf a$};

\draw (3.2,0.9) node[rotate=-30,circle,draw] {$\mathbf E$};

\draw (2.97,0.72) node[rotate=-30,circle,draw] {$\mathbf {b}$};

\draw (2.67,0.61) node[rotate=-30,circle,draw] {$\mathbf {F\sharp}$};

\draw (2.32,0.61) node[rotate=-30,circle,draw] {$\mathbf {c\sharp}$};

\draw (2.02,0.72) node[rotate=-30,circle,draw] {$\mathbf {A\flat}$};

\draw (1.87,0.9) node[rotate=-30,circle,draw] {$\mathbf e\flat$};

\draw (1.87,1.1) node[rotate=-30,circle,draw] {$\mathbf B\flat$};

\draw (2.02,1.27) node[rotate=-30,circle,draw] {$\mathbf f$};

\draw (2.35,1.35) node[rotate=-30,circle,draw] {$\mathbf C$};
\draw (2.67,1.35) node[rotate=-30,circle,draw] {$\mathbf g$};

\draw (3.17,0.36) node[rotate=-30,rectangle,fill=white!10] {$G\flat$};
\draw (3.38,0.14) node[rotate=-30,rectangle,fill=white!10] {$C$};

\draw (3.23,0.62) node[rotate=-30] {$D$};
\draw (3.46,0.84) node[rotate=-30] {$A\flat$};

\draw (3.75,0) node[rotate=-30,rectangle,fill=white!10] {$E$};
\draw (4.25,0) node[rotate=-30,rectangle,fill=white!10] {$A\sharp$};
\draw (3.75,1) node[rotate=-30,rectangle,fill=white!10] {$E$};
\draw (4.25,1) node[rotate=-30,rectangle,fill=white!10] {$B\flat$};
\draw (4.65,0.87) node[rotate=-30,rectangle,fill=white!10] {$F\sharp$};
\draw (4.85,0.64) node[rotate=-30,rectangle,fill=white!10] {$C$};
\draw (4.85,0.37) node[rotate=-30,rectangle,fill=white!10] {$G\sharp$};

\draw (3.65,0.25) node[rotate=-30] {$A\flat$};

\draw (3.77,0.2) node[rotate=-30] {$G\sharp$};
\draw (3.55,0.36) node[rotate=-30,rectangle,fill=white!10] {$E\flat$};
\draw (4,0.19) node[rotate=-30,rectangle,fill=white!10] {$C\sharp$};
\draw (4.25,0.24) node[rotate=-30,rectangle,fill=white!10] {$F\sharp$};
\draw (4.53,0.33) node[rotate=-30,rectangle,fill=white!10] {$B$};
\draw (4.62,0.15) node[rotate=-30,rectangle,fill=white!10] {$D$};
\draw (4.6,0.5) node[rotate=-30,rectangle,fill=white!10] {$E$};
\draw (3.5,0.5) node[rotate=-30,rectangle,fill=white!10] {$B\flat$};
\draw (3.6,0.65) node[rotate=-30,rectangle,fill=white!10] {$F$};
\draw (3.8,0.73) node[rotate=-30,rectangle,fill=white!10] {$C$};
\draw (4,0.76) node[rotate=-30,rectangle,fill=white!10] {$G$};
\draw (4.23,0.73) node[rotate=-30,rectangle,fill=white!10] {$D$};
\draw (4.45,0.65) node[rotate=-30,rectangle,fill=white!10] {$A$};
\end{tikzpicture}};
\end{tikzpicture}
\end{center}
\caption{A major/minor edge-tonnetz of type $G_2$ with fundamental alcove shaded.}\label{f:G2-full}
\end{figure}

\subsection{Vertices and Symmetry} \label{s:G2vertices}
The $G_2$ tonnetz on a torus has three types of vertices. 
\begin{itemize}
\item There is one vertex in the center that has all $12$ pitch-classes attached to it. 
\item There are two vertices that come from the corners of the hexagon (two that remain after gluing). These are labelled by augmented triads, one by $C, E, G\sharp$, and of the other by $D, F\sharp,A\sharp$. (To be formally precise, the label is of course a multiset, e.g. $\{C,C,E,E,G\sharp,G\sharp\}$ in the first case, since there are six edges contributing notes to the vertex.) 
\item Finally, there are three vertices that are $4$-valent. Each one of these is labelled by one of the three diminished seventh chords. 
\end{itemize}

Before gluing we can observe the following rotational symmetries of the tonnetz in the plane.
\begin{itemize}
\item 
Rotating  clockwise by $60^\circ$ about the central vertex amounts to transposing up by a whole tone. 
\item
Rotating clockwise by $120^\circ$ about the corners of the fundamental hexagon corresponds to transposing up by a major third. 
\item 
For each $4$-valent vertex, rotation by $180^\circ$ corresponds to transposition by a tritone.  
\end{itemize}
We note that the order of the rotational symmetry about a vertex is always half the valency here.
\subsection{Dynkin diagrams and the Langlands dual $G_2$ tonnetz} The fundamental alcove for the $G_2$ tonnetz in Figure~\ref{f:G2-full} is the one containing the major $D$ triad, as indicated in the figure. The edges corresponding to the reflection hyperplanes of the finite Weyl group $G_2$ are labelled $A$ and $D$. This leads to the following labelling of the $G_2$ Dynkin diagram,
\begin{center}
\quad \dynkin[edge length=.8cm,
root radius=.09cm,labels={A,D}]{G}{2}.
\end{center} 

\begin{figure}
\centering
\includegraphics[scale=0.4,clip,trim=0.5cm 0cm 0.5cm 0.5cm]{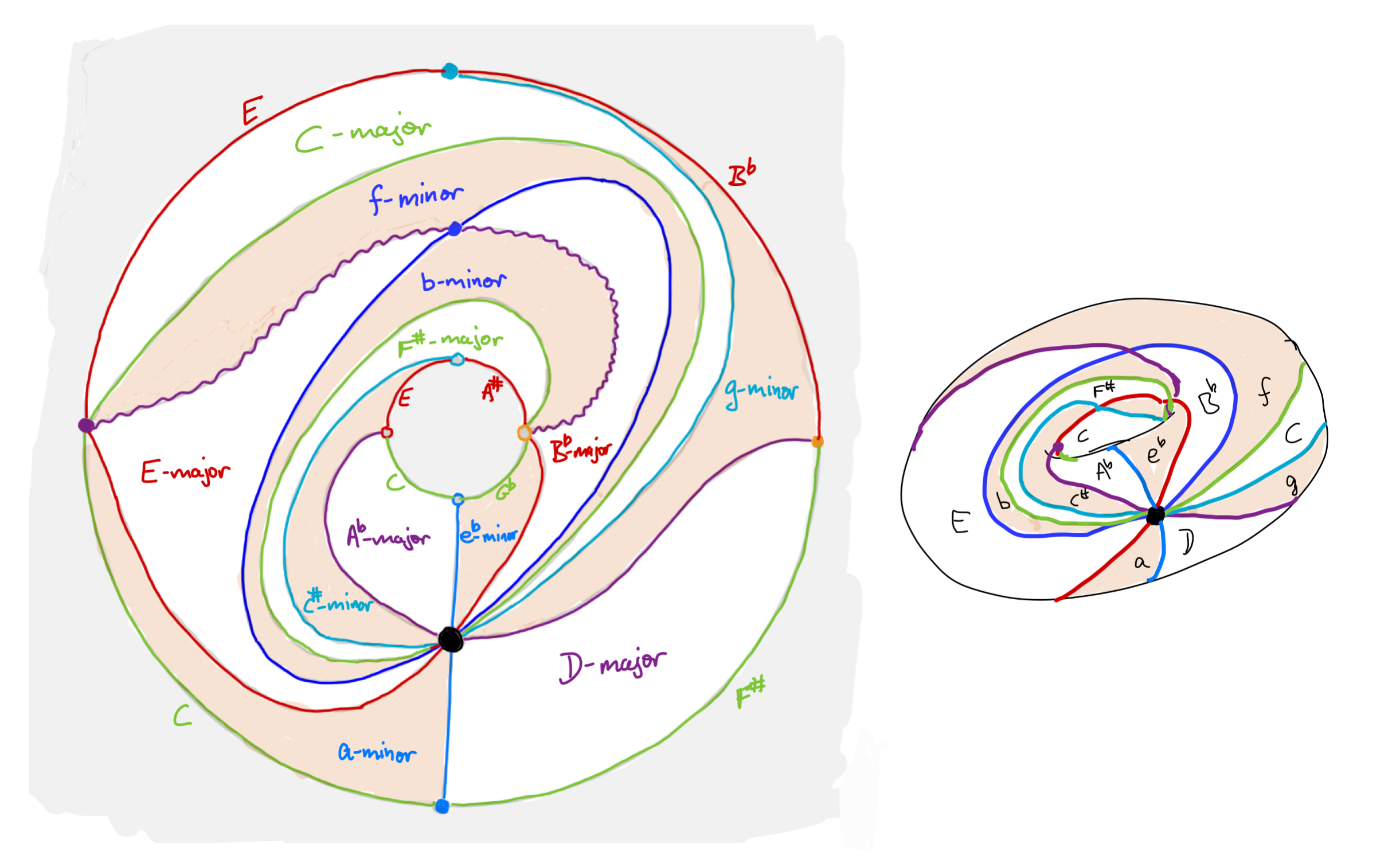}
\caption{A realisation of the first major/minor tonnetz of type $G_2$ on a torus. The intermediate image on the left is the cylinder obtained by identifying the purple $A\flat/D$ and $G\sharp/D$ sides in Figure~\ref{f:G2-full} (flattend to an annulus). Gluing the inner to the outer circle gives the torus on the right-hand side.}\label{f:G2ontorus}
\end{figure}

The Langlands dual $G_2$ tonnetz is now obtained by `swapping long and short sides' of our fundamental alcove. Namely, as in Section~\ref{s:C2}, we alter side lengths,  turning the side of the fundamental alcove that is labelled $D$ into the shorter side and $A$ into the longer one. Then we get a right angle at the end of the side labelled $D$ which means we should put an $F$ on the remaining side, since right angles correspond to minor thirds. This is illustrated in Figure~\ref{f:G2-v2}. The $D$-major triad has become $D$-minor, and all other triads have similarly reversed their quality. 

In terms of the Dynkin diagram, we could represent this new tonnetz by changing the direction of the arrow from the old one, as happened in the change from $B_2$ to $C_2$. But this is not conventional for $G_2$, and so instead we just switch around the labels,
\begin{center}
\dynkin[edge length=.8cm,
root radius=.09cm,labels={D,A}]{G}{2}.\qquad
\end{center}
The analysis of the vertices and symmetries of this tonnetz is exactly the same as in Section~\ref{s:G2vertices}, with the one change that the augmented triads attached to the vertices that are the corners of the fundamental hexagons are now $F,A,C\sharp$ and $G,B,D\sharp$, the two that were left out before. 

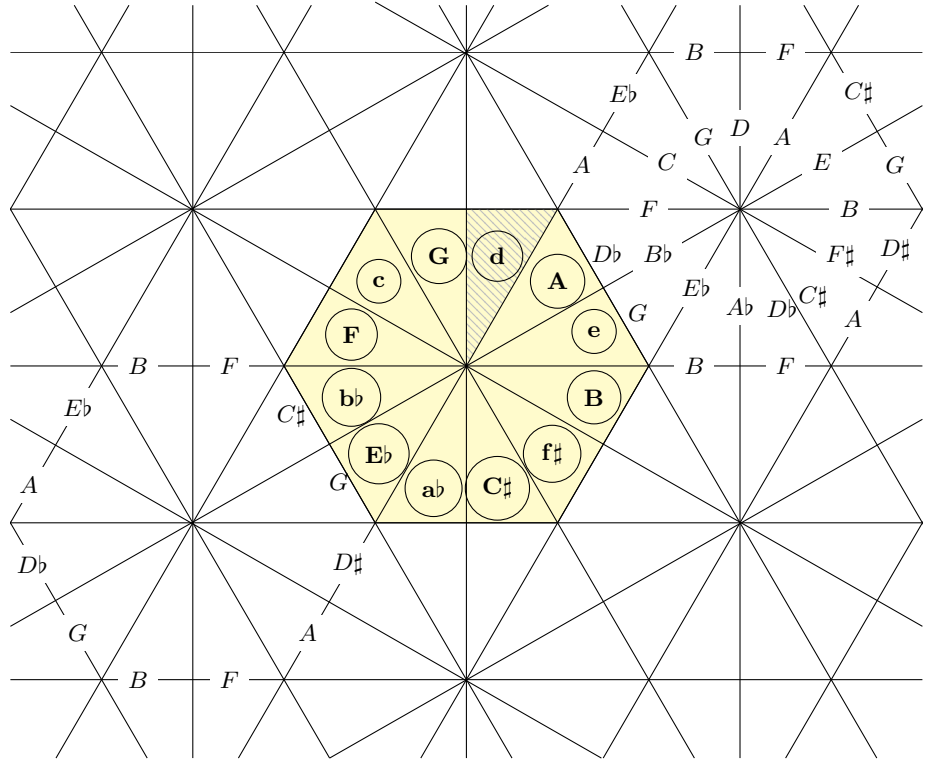
\begin{figure}
\begin{center}
\begin{tikzpicture}[scale=0.8,
  xscale=3,yscale=5.199,
  note/.style={
  minimum size=0.75cm},
  every node/.append style={font=\footnotesize}
  ]

\draw[fill=yellow!25] (2,0.5) -- (1.5,1) -- (2,1.5) -- (3,1.5) -- (3.5,1)--(3,0.5)-- cycle;

\fill[ pattern=north west lines, pattern color=lightgray]  (2.5,1.5) -- (3,1.5) -- (2.5,1)-- cycle;




\begin{scope}
\newcommand*\columns{5}
\newcommand*\rows{2}
\clip(0,-\pgflinewidth) rectangle (\columns,\rows);
\foreach \y in {0,0.5,1,...,\rows} 
  \draw (0,\y) -- (\columns,\y);
\foreach \z in {-1.5,-0.5,...,\columns} 
{
    \draw (\z,\rows) -- (\z+2,0);
    \draw (\z,0) -- (\z+2,\rows);
}
\end{scope}

\draw (4,2.15)--(4,-0.25);
\draw (2.5,2.15)--(2.5,-0.25);
\draw (1,2.15)--(1,-0.25);
\draw (0,0.17)--(5,1.83);
\draw (0,1.17)--(2.95,2.15);
\draw (1.75,-0.25)--(5,0.83);
\draw (0,0.83)-- (3.24,-0.25);
\draw (0,1.83)-- (5,0.17);
\draw (2,2.16)-- (5,1.17);

\draw (4.5,2)--(4.65,2.15);
\draw (4.5,2)--(4.35,2.15);
\draw (3.5,2)--(3.65,2.15);
\draw (3.5,2)--(3.35,2.15);
\draw (2.5,2)--(2.65,2.15);
\draw (2.5,2)--(2.35,2.15);
\draw (1.5,2)--(1.65,2.15);
\draw (1.5,2)--(1.35,2.15);
\draw (0.5,2)--(0.65,2.15);
\draw (0.5,2)--(0.35,2.15);

\draw (4.5,0)--(4.75,-0.25);
\draw (4.5,0)--(4.25,-0.25);
\draw (3.5,0)--(3.75,-0.25);
\draw (3.5,0)--(3.25,-0.25);
\draw (2.5,0)--(2.75,-0.25);
\draw (2.5,0)--(2.25,-0.25);
\draw (1.5,0)--(1.75,-0.25);
\draw (1.5,0)--(1.25,-0.25);
\draw (0.5,0)--(0.75,-0.25);
\draw (0.5,0)--(0.25,-0.25);


\draw (0.7,1) node [rectangle,fill=white!10] {$B$};
\draw (1.2,1) node [rectangle,fill=white!10] {$F$};

\draw (0.37,0.87) node [rectangle,fill=white!10] {$E\flat$};
\draw (1.54,0.84) node  {$C\sharp$};

\draw (0.1,0.62) node [rectangle,fill=white!10] {$A$};
\draw (1.8,0.63) node  {$G$};

\draw (0.12,0.37) node [rectangle,fill=white!10] {$D\flat$};
\draw (1.85,0.37) node [rectangle,fill=white!10] {$D\sharp$};

\draw (0.37,0.15) node [rectangle,fill=white!10] {$G$};
\draw (1.63,0.15) node   [rectangle,fill=white!10] {$A$};

\draw (0.7,0) node [rectangle,fill=white!10] {$B$};
\draw (1.2,0) node [rectangle,fill=white!10] {$F$};

]


\draw (3,1.27) node[circle,draw] {$\mathbf A$};

\draw (3.2,1.11) node[circle,draw] {$\mathbf e$};

\draw (3.2,0.9) node[circle,draw] {$\mathbf B$};

\draw (2.97,0.72) node[circle,draw] {$\mathbf {f\sharp}$};

\draw (2.67,0.61) node[circle,draw] {$\mathbf {C\sharp}$};

\draw (2.32,0.61) node[circle,draw] {$\mathbf {a\flat}$};

\draw (2.02,0.72) node[circle,draw] {$\mathbf {E\flat}$};

\draw (1.87,0.9) node[circle,draw] {$\mathbf b\flat$};

\draw (1.87,1.1) node[circle,draw] {$\mathbf F$};

\draw (2.02,1.27) node[circle,draw] {$\mathbf c$};

\draw (2.35,1.35) node[circle,draw] {$\mathbf G$};
\draw (2.67,1.35) node[circle,draw] {$\mathbf d$};

\draw (3.27,1.36) node {$D\flat$};
\draw (3.44,1.17) node {$G$};

\draw (3.13,1.64) node[rectangle,fill=white!10] {$A$};
\draw (3.36,1.87) node[rectangle,fill=white!10] {$E\flat$};
\draw (3.75,1) node[rectangle,fill=white!10] {$B$};
\draw (4.25,1) node[rectangle,fill=white!10] {$F$};
\draw (3.75,2) node[rectangle,fill=white!10] {$B$};
\draw (4.25,2) node[rectangle,fill=white!10] {$F$};
\draw (4.65,1.87) node[rectangle,fill=white!10] {$C\sharp$};
\draw (4.85,1.64) node[rectangle,fill=white!10] {$G$};
\draw (4.85,1.37) node[rectangle,fill=white!10] {$D\sharp$};

\draw (3.76,1.25) node[rectangle,fill=white!10] {$E\flat$};
\draw (3.55,1.36) node[rectangle,fill=white!10] {$B\flat$};
\draw (4,1.19) node[rectangle,fill=white!10] {$A\flat$};
\draw (4.4,1.22) node {$C\sharp$};
\draw (4.23,1.185) node {$D\flat$};
\draw (4.55,1.35) node[rectangle,fill=white!10] {$F\sharp$};
\draw (4.62,1.15) node[rectangle,fill=white!10] {$A$};
\draw (4.6,1.5) node[rectangle,fill=white!10] {$B$};
\draw (3.5,1.5) node[rectangle,fill=white!10] {$F$};
\draw (3.6,1.65) node[rectangle,fill=white!10] {$C$};
\draw (3.8,1.73) node[rectangle,fill=white!10] {$G$};
\draw (4,1.76) node[rectangle,fill=white!10] {$D$};
\draw (4.23,1.73) node[rectangle,fill=white!10] {$A$};
\draw (4.45,1.65) node[rectangle,fill=white!10] {$E$};
\end{tikzpicture}
\end{center}
\caption{The `Langlands dual' major/minor edge-tonnetz of type $G_2$. The fundamental alcove is the one containing the minor $D$ chord.}\label{f:G2-v2}
\end{figure}

Finally, we note that each of the two edge tonnetzes constructed in this section is a major/minor tonnetz, but it is not complete. Nevertheless, the two taken together form a (disconnected) complete major/minor edge tonnetz.

\section{Two major tritone-edge tonnetz examples on the $A_2$ tiling of the plane}\label{s:tritone}

So far we have been focusing on constructing edge-tonnetz examples. We now construct a new kind of tonnetz in which every edge has two notes associated to it. Namely, in this section we give two examples of a major tritone-edge tonnetz. The chords are all the major triads whose base notes run through a whole-tone scale, and every edge individually has two notes that are a tritone apart. Which note belongs to which triad is illustrated by which side of the edge it is written on. Similarly, which vertex a note is belongs to is illustrated by proximity.

\subsection{}\label{s:tritone1} The first example is given in Figure~\ref{f:tritone1}, and it is the simpler of the two. It is based on a single hexagon which repeats, periodically tiling the plane. The three notes of the triad associated to a triangle are always read off from the base note going around the edges in a clockwise orientation. In the central hexagon in the figure, we have indicated which major triad is represented by each triangle. Note how the edge labelled $A$ and $D^\sharp$ lies between $A$-major and $B$-major. It has $A$ above the line as part of the $A$-major triad and $D^\sharp$ below the line as part of the $B$-major triad. Also indicated by the placement of the notes is which vertex they map to. The full data of this tonnetz includes a set of six notes associated to each vertex, a tritone pair associated to each edge, and triad associated to each triangular face (all compatible with one another).  

This tonnetz has a variety of symmetries. Clockwise rotation by $120^\circ$ around any vertex corresponds to raising every note by a major third. Moreover clockwise rotation by $60^\circ$ around the central vertex of the hexagon  corresponds to raising every note by a whole-tone. 

We can identify opposite edges of the fundamental hexagon, highlighted in Figure~\ref{f:tritone1}, and obtain a finite tonnetz on a torus triangulated by $6$ triangles. The resulting tonnetz on the torus has only three vertices. One is the one in the center of the fundamental hexagon, and it has associated to it the pitch classes of the whole tone scale: $A, B,D\flat,E\flat,F,G$. The others each have the union of two augmented triads allocated to them, $\{ A, B\flat, C\sharp, D, F, F\sharp\}$ and  $\{G,A\flat, B,C,D\sharp, E\}$.

Note that there are only $6$ distinct major triads in this tonnetz. Transposing every note up by a semitone would give a tonnetz that covers the remaining $6$ major triads.

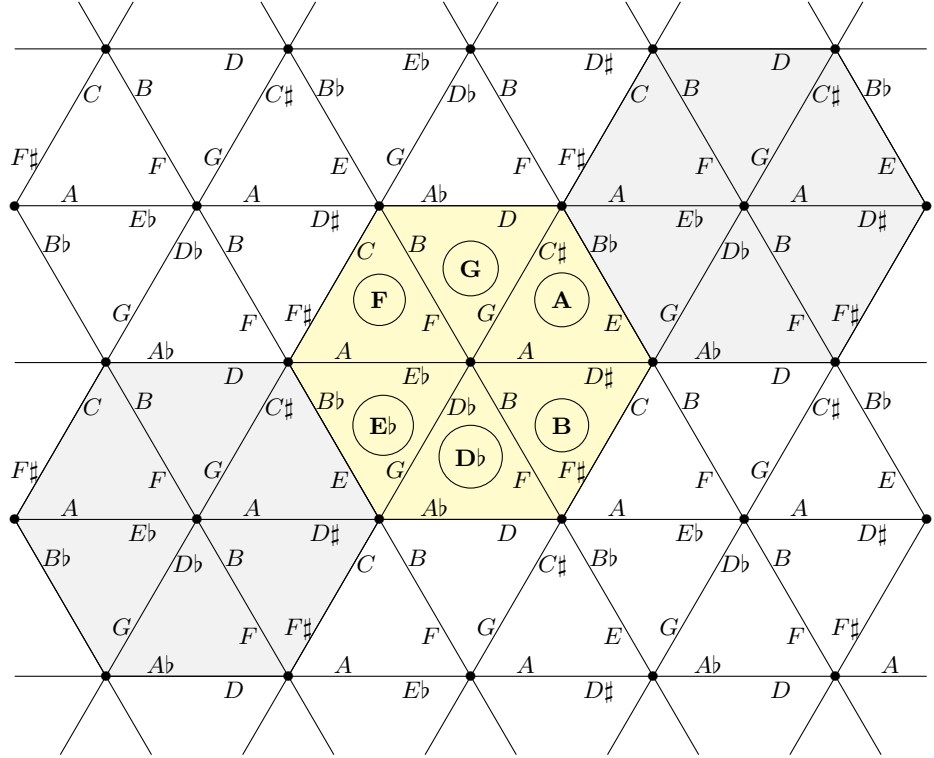
\begin{figure}
\begin{center}
\begin{tikzpicture}[scale=0.8,
  xscale=3,yscale=5.199,
  note/.style={
  minimum size=0.75cm},
  every node/.append style={font=\footnotesize}
  ]

\draw[fill=yellow!25] (2,0.5) -- (1.5,1) -- (2,1.5) -- (3,1.5) -- (3.5,1)--(3,0.5)-- cycle;

\draw[fill=gray!10] (2,0.5) -- (1.5,1) -- (0.5,1) -- (0,0.5) -- (0.5,0)--(1.5,0)-- cycle;


\draw[fill=gray!10] (3,1.5) -- (3.5,2) -- (4.5,2) -- (5,1.5) -- (4.5,1)--(3.5,1)-- cycle;

\foreach [count=\row from 0] \notelist in {
{$\ G$,$F\sharp\,$,$\ G$,$\ G$,$F\sharp\,$},
{$F\sharp\,$,$\ G$,$\ G$,$F\sharp\,$,$\ G$},
  {$\ G$,$F\sharp\,$,$\ G$,$\ G$,$F\sharp\,$},
{$F\sharp\,$,$\ G$,$\ G$,$F\sharp\,$,$\ G$},
  {$\ $,$\ $,$\ $,$\ $,$\ $}}
  \foreach \edgenote [count=\column from 0,evaluate={\colX=\column+0.5-mod(\row,2)/2;}] in \notelist
     \node [note] at (\colX+.065,\row*0.5+.15) {\strut \edgenote};

\foreach [count=\row from 0] \notelist in {
{$D$,$E\flat$,$D\sharp$,$D$},
{$E\flat$,$D\sharp$,$D$,$E\flat$,$D\sharp$},
{$D$,$E\flat$,$D\sharp$,$D$},
{$E\flat$,$D\sharp$,$D$,$E\flat$,$D\sharp$},
{$D$,$E\flat$,$D\sharp$,$D$}}
  \foreach \edgenote [count=\column from 0,evaluate={\colX=\column+0.5-mod(\row,2)/2;}] in \notelist
     \node [note] at (\colX+.7,\row*0.5-.05) {\strut \edgenote};

\foreach [count=\row from 0] \notelist in {
{$A\flat$,$A$,$A$,$A\flat$,$A$},
{$A$,$A$,$A\flat$,$A$,$A$},
{$A\flat$,$A$,$A$,$A\flat$},
{$A$,$A$,$A\flat$,$A$,$A$}}
  \foreach \edgenote [count=\column from 0,evaluate={\colX=\column+0.5-mod(\row,2)/2;}] in \notelist
     \node [note] at (\colX+.3,\row*0.5+.03) {\strut \edgenote};

\foreach [count=\row from 0] \notelist in {
{$D\flat$,$C\ $,$C\sharp$,$D\flat$},
{$C\ $,$C\sharp$,$D\flat$,$C\ $,$C\sharp$},
{$D\flat$,$C\ $,$C\sharp$,$D\flat$},
{$C\ $,$C\sharp$,$D\flat$,$C\ $,$C\sharp$}}
  \foreach \edgenote [count=\column from 0,evaluate={\colX=\column+0.5-mod(\row,2)/2;}] in \notelist
     \node [note] at (\colX+.45,\row*0.5+.35) {\strut \edgenote};

\foreach [count=\row from 0] \notelist in {
{$B\flat$,$B\ $,$B\ $,$B\flat$,$B\ $},
{\ ,$B\ $,$B\flat$,$B\ $,$B\ $,$B\flat$},
{$B\flat$,$B\ $,$B\ $,$B\flat$,$B\ $},
{\ ,$B\ $,$B\flat$,$B\ $,$B\ $,$B\flat$}}
  \foreach \edgenote [count=\column from 0,evaluate={\colX=\column+0.5-mod(\row,2)/2;}] in \notelist
     \node [note] at (\colX+-.27,\row*0.5+.37) {\strut \edgenote};

\foreach [count=\row from 0] \notelist in {
{$F$,$F$,$E$,$F$},
{$F$,$E$,$F$,$F$,$E$},
{$F$,$F$,$E$,$F$},
{$F$,$E$,$F$,$F$,$E$}}
  \foreach \edgenote [count=\column from 0,evaluate={\colX=\column+0.5-mod(\row,2)/2;}] in \notelist
     \node [note] at (\colX+.78,\row*0.5+.12) {\strut \edgenote};

\begin{scope}
\newcommand*\columns{5}
\newcommand*\rows{2}
\clip(0,-\pgflinewidth) rectangle (\columns,\rows);
\foreach \y in {0,0.5,1,...,\rows} 
  \draw (0,\y) -- (\columns,\y);
\foreach \z in {-1.5,-0.5,...,\columns} 
{
    \draw (\z,\rows) -- (\z+2,0);
    \draw (\z,0) -- (\z+2,\rows);
}
\end{scope}

\draw (4.5,2)--(4.65,2.15);
\draw (4.5,2)--(4.35,2.15);
\draw (3.5,2)--(3.65,2.15);
\draw (3.5,2)--(3.35,2.15);
\draw (2.5,2)--(2.65,2.15);
\draw (2.5,2)--(2.35,2.15);
\draw (1.5,2)--(1.65,2.15);
\draw (1.5,2)--(1.35,2.15);
\draw (0.5,2)--(0.65,2.15);
\draw (0.5,2)--(0.35,2.15);

\draw (4.5,0)--(4.75,-0.25);
\draw (4.5,0)--(4.25,-0.25);
\draw (3.5,0)--(3.75,-0.25);
\draw (3.5,0)--(3.25,-0.25);
\draw (2.5,0)--(2.75,-0.25);
\draw (2.5,0)--(2.25,-0.25);
\draw (1.5,0)--(1.75,-0.25);
\draw (1.5,0)--(1.25,-0.25);
\draw (0.5,0)--(0.75,-0.25);
\draw (0.5,0)--(0.25,-0.25);


\foreach [count=\row from 0] \notelist in {
{$\bullet$,$\bullet$,$\bullet$,$\bullet$,$\bullet$},
  {$\bullet$,$\bullet$,$\bullet$,$\bullet$,$\bullet$,$\bullet$},
  {$\bullet$,$\bullet$,$\bullet$,$\bullet$,$\bullet$},
  {$\bullet$,$\bullet$,$\bullet$,$\bullet$,$\bullet$,$\bullet$},
  {$\bullet$,$\bullet$,$\bullet$,$\bullet$,$\bullet$}}
  \foreach \note [count=\column from 0,evaluate={\colX=\column+0.5-mod(\row,2)/2;}] in \notelist
     \node [note] at (\colX,\row*0.5) {\strut \note};



\draw (3,1.2) node[circle,draw] {$\mathbf A$};

\draw (2.5,1.3) node[circle,draw] {$\mathbf G$};

\draw (2,1.2) node[circle,draw] {$\mathbf F$};

\draw (2.02,0.8) node[circle,draw] {$\mathbf {E\flat}$};

\draw (2.5,0.7) node[circle,draw] {$\mathbf {D\flat}$};

\draw (3,0.8) node[circle,draw] {$\mathbf B$};

\end{tikzpicture}
\end{center}
\caption{A major tritone edge-tonnetz with a fundamental region consisting of $6$ triangles indicated in the center.}\label{f:tritone1}
\end{figure}

\subsection{}\label{s:tritone2} To illustrate the fact that the allocation of notes to vertices is a part of the tonnetz that shouldn't be overlooked, we construct another tonnetz with exactly the same chords but that allocates notes to vertices differently and as a result also has different symmetries, see Figure~\ref{f:tritone2}.
\begin{figure}
\begin{center}
\begin{tikzpicture}[scale=0.8,
  xscale=3,yscale=5.199,
  note/.style={
  minimum size=0.75cm},
  every node/.append style={font=\footnotesize}
  ]

\draw[fill=yellow!25] (0.5,1) -- (1.5,2) -- (3.5,2) -- (4.5,1) -- (3.5,0)--(1.5,0)-- cycle;

\foreach [count=\row from 0] \notelist in {
{},
{,$\ G$,$\ G$,$F\sharp$,$\ G$},
  {},
{,$\ G$,$\ G$,$F\sharp$,$\ G$},
  {$\ $,$\ $,$\ $,$\ $,$\ $}}
  \foreach \edgenote [count=\column from 0,evaluate={\colX=\column+0.5-mod(\row,2)/2;}] in \notelist
     \node [note] at (\colX+.065,\row*0.5+.15) {\strut \edgenote};

\foreach [count=\row from 0] \notelist in {
{$\ G$,$F\sharp$,$\ G$,$\ G$,$F\sharp$},
{,,,,},
  {$\ G$,$F\sharp$,$\ G$,$\ G$,$F\sharp$},
{,,,,},
  {$\ $,$\ $,$\ $,$\ $,$\ $}}
  \foreach \edgenote [count=\column from 0,evaluate={\colX=\column+0.5-mod(\row,2)/2;}] in \notelist
     \node [note] at (\colX+.26,\row*0.5+.37) {\strut \edgenote};

\foreach [count=\row from 0] \notelist in {
{$\qquad\qquad D$,,$E\flat\qquad\qquad D\sharp$,,$D\qquad\qquad$},
{$\qquad\qquad E\flat$,,$D\sharp\qquad\qquad D$,,$E\flat\qquad\qquad D\sharp$,},
{,$D\qquad\qquad E\flat$,,$D\sharp\qquad\qquad D$,},
{,$E\flat\qquad\qquad D\sharp$,,$D\qquad\qquad E\flat$,,$D\sharp\qquad\qquad$},
{$\qquad\qquad D$,,$E\flat\qquad\qquad D\sharp$,,$D\qquad\qquad$}}
  \foreach \edgenote [count=\column from 0,evaluate={\colX=\column+0.5-mod(\row,2)/2;}] in \notelist
     \node [note] at (\colX,\row*0.5-.05) {\strut \edgenote};

\foreach [count=\row from 0] \notelist in {
{{},$A\flat\qquad\qquad A$,{},$A\qquad\qquad A\flat$,{}},
{{},$A\qquad\qquad A $,{},$A\flat \qquad\qquad A$,{},$A\qquad$},
{$A\qquad\qquad A\flat$,{},$A\ \qquad\qquad A$,{},$A\flat \qquad\qquad A $},
{$\qquad\qquad A$,{},$A\qquad\qquad A\flat$,{},$A\qquad\qquad A$}}
  \foreach \edgenote [count=\column from 0,evaluate={\colX=\column+0.5-mod(\row,2)/2;}] in \notelist
     \node [note] at (\colX,\row*0.5+.035) {\strut \edgenote};

\foreach [count=\row from 0] \notelist in {
{{},{},{},{}},
{$C\ $,$C\sharp$,$D\flat$,$C\ $,$C\sharp$},
{{},{},{},{}},
{$C\ $,$C\sharp$,$D\flat$,$C\ $,$C\sharp$}}
  \foreach \edgenote [count=\column from 0,evaluate={\colX=\column+0.5-mod(\row,2)/2;}] in \notelist
     \node [note] at (\colX+.435,\row*0.5+.34) {\strut \edgenote};

\foreach [count=\row from 0] \notelist in {
{$D\flat$,$C\ $,$C\sharp$,$D\flat$},
{{},{},{},{},{}},
{$D\flat$,$C\ $,$C\sharp$,$D\flat$,$C\ $},
{{},{},{},{},{}}}
  \foreach \edgenote [count=\column from 0,evaluate={\colX=\column+0.5-mod(\row,2)/2;}] in \notelist
     \node [note] at (\colX+.22,\row*0.5+.12) {\strut \edgenote};

\foreach [count=\row from 0] \notelist in {
{,$B\ $,,$B\flat$,},
{\ ,,$B\flat$,,$B\ $,},
{$B\flat$,,$B\ $,,$B\ $},
{\ ,$B\ $,,$B\ $,,$B\flat$}}
  \foreach \edgenote [count=\column from 0,evaluate={\colX=\column+0.5-mod(\row,2)/2;}] in \notelist
     \node [note] at (\colX+-.27,\row*0.5+.37) {\strut \edgenote};

\foreach [count=\row from 0] \notelist in {
{$B\flat$,,$B\ $,,$B\ $},
{\ ,$B\ $,,$B\ $,,$B\flat$},
{,$B\ $,,$B\flat$,},
{\ ,,$B\flat$,,$B\ $,}}
  \foreach \edgenote [count=\column from 0,evaluate={\colX=\column+0.5-mod(\row,2)/2;}] in \notelist
     \node [note] at (\colX+-.058,\row*0.5+.15) {\strut \edgenote};

\foreach [count=\row from 0] \notelist in {
{{},$F$,{},$E$,{}},
{{},{},$E$,{},$F$,{}},
{{$E$},{},$F$,{},$F$},
{{},$F$,,$F$,,$E$}}
  \foreach \edgenote [count=\column from 0,evaluate={\colX=\column+0.5-mod(\row,2)/2;}] in \notelist
     \node [note] at (\colX+-.22,\row*0.5+.12) {\strut \edgenote};

\foreach [count=\row from 0] \notelist in {
{,$F$,,$F$},
{$F$,,$F$,,$E$},
{$F$,,$E$,},
{,$E$,,$F$,}}
  \foreach \edgenote [count=\column from 0,evaluate={\colX=\column+0.5-mod(\row,2)/2;}] in \notelist
     \node [note] at (\colX+.57,\row*0.5+.34) {\strut \edgenote};

\begin{scope}
\newcommand*\columns{5}
\newcommand*\rows{2}
\clip(0,-\pgflinewidth) rectangle (\columns,\rows);
\foreach \y in {0,0.5,1,...,\rows} 
  \draw (0,\y) -- (\columns,\y);
\foreach \z in {-1.5,-0.5,...,\columns} 
{
    \draw (\z,\rows) -- (\z+2,0);
    \draw (\z,0) -- (\z+2,\rows);
}
\end{scope}

\draw (4.5,2)--(4.65,2.15);
\draw (4.5,2)--(4.35,2.15);
\draw (3.5,2)--(3.65,2.15);
\draw (3.5,2)--(3.35,2.15);
\draw (2.5,2)--(2.65,2.15);
\draw (2.5,2)--(2.35,2.15);
\draw (1.5,2)--(1.65,2.15);
\draw (1.5,2)--(1.35,2.15);
\draw (0.5,2)--(0.65,2.15);
\draw (0.5,2)--(0.35,2.15);

\draw (4.5,0)--(4.75,-0.25);
\draw (4.5,0)--(4.25,-0.25);
\draw (3.5,0)--(3.75,-0.25);
\draw (3.5,0)--(3.25,-0.25);
\draw (2.5,0)--(2.75,-0.25);
\draw (2.5,0)--(2.25,-0.25);
\draw (1.5,0)--(1.75,-0.25);
\draw (1.5,0)--(1.25,-0.25);
\draw (0.5,0)--(0.75,-0.25);
\draw (0.5,0)--(0.25,-0.25);


\foreach [count=\row from 0] \notelist in {
{$\bullet$,$\bullet$,$\bullet$,$\bullet$,$\bullet$},
  {$\bullet$,$\bullet$,$\bullet$,$\bullet$,$\bullet$,$\bullet$},
  {$\bullet$,$\bullet$,$\bullet$,$\bullet$,$\bullet$},
  {$\bullet$,$\bullet$,$\bullet$,$\bullet$,$\bullet$,$\bullet$},
  {$\bullet$,$\bullet$,$\bullet$,$\bullet$,$\bullet$}}
  \foreach \note [count=\column from 0,evaluate={\colX=\column+0.5-mod(\row,2)/2;}] in \notelist
     \node [note] at (\colX,\row*0.5) {\strut \note};



\draw (3,1.2) node[circle,draw] {$\mathbf A$};

\draw (2.5,1.3) node[circle,draw] {$\mathbf G$};

\draw (2,1.2) node[circle,draw] {$\mathbf F$};

\draw (2.02,0.8) node[circle,draw] {$\mathbf {E\flat}$};

\draw (2.5,0.7) node[circle,draw] {$\mathbf {D\flat}$};

\draw (3,0.8) node[circle,draw] {$\mathbf B$};

\end{tikzpicture}
\end{center}
\caption{The periodic major tritone-edge tonnetz with a fundamental domain consisting of $24$ triangles described in Section~\ref{s:tritone2}. This fundamental domain also differs from the one in Figure~\ref{f:Euler}.}\label{f:tritone2}
\end{figure}
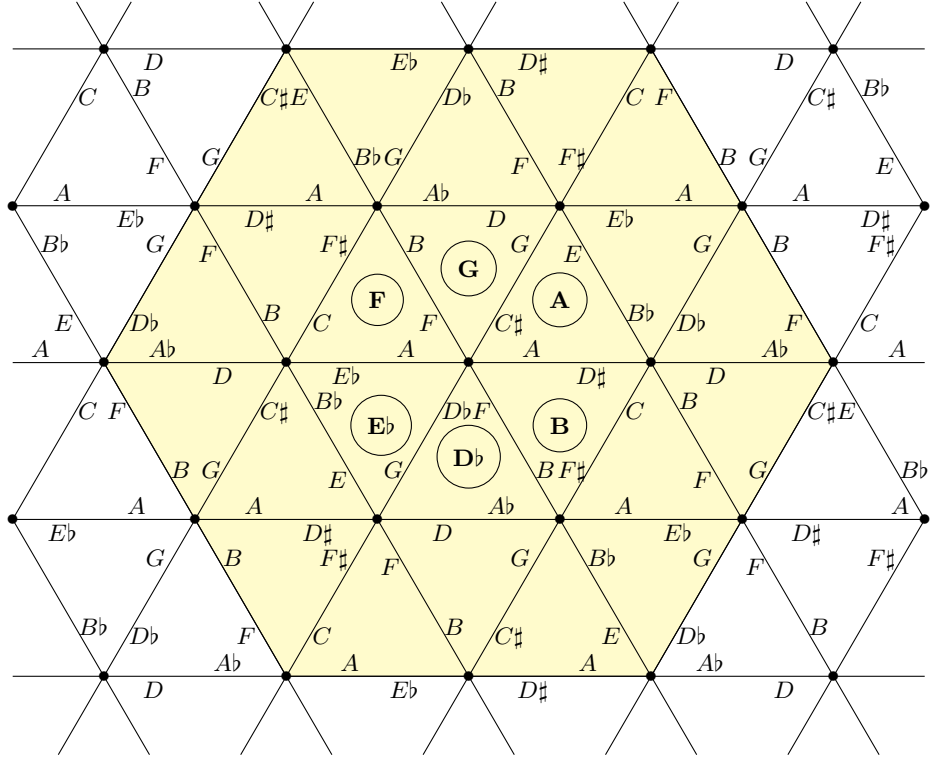
We also give a more fanciful illustration of it to help with visualisation, in  Figure~\ref{f:tritone2hand}. 

We now describe the structure of this tonnetz in a bit more detail. Which notes are associated to a given vertex is indicated in Figure~\ref{f:tritone2hand} by the highlighted `petals' and the stick-figures. For example, the petal vertices are as follows.  
\begin{itemize}
\item For the vertex in the middle $\tonnetz(\rho)=\{A,C^\sharp,F,A,D^\flat,F\}$. Since $D^{\flat}$ and $C^\sharp$ are identified, the underlying set has cardinality $3$ and represents an augmented triad: $A, C^\sharp, F$. \item 
The petal vertex at the far right on the other hand has $\{A,C^\sharp,F,A^\flat, C, E\}$ associated to it (an augmented triad and its translation by a semitone). 
\end{itemize}
These are the only two possibilities arise for the petal vertices, as is easy to check. A stick figure at a vertex shows when \textit{two} sides of a single triangle contribute their notes to that vertex (if the triangle contains the head of the figure), and when just one of the sides contributes (if there is an arm or a leg) and when none of them do. For such a vertex either the six associated notes will be clustered or there will be three notes separated by whole-tones. For example we have the following.
\begin{itemize}
\item  Associated to the vertex to the right of the central vertex, we have $\{B^\flat, B,C,D^\flat,D,D^\sharp\}$.
The $B$ and $D$ come from the $G$-major triangle containing the head, the $B^\flat$ and $D^\sharp$ come from the legs of the stick figure. 
\item 
The upright stick figure below and to the left of the previous one has $\{F^\sharp, G, A^\flat, A, B^\flat, B\}$ associated to it, with $F^\sharp$ and $B$ coming from the $B$-major triangle at the head, while the legs contribute $G$ and $B^\flat$. 
\item 
The stick-figure vertices along the outside all have triads consisting of whole-tone steps associated to them. For example, the multi-set associated to the upper right-hand stick-figure vertex is $\{G,G,A,A,B,B\}$. 
\end{itemize}
We also note the following rotational symmetry. Namely, clockwise rotation by $120^\circ$ about any `petal' vertex corresponds to raising every note by a major third. 

\begin{figure}\centering
\includegraphics[scale=0.17]{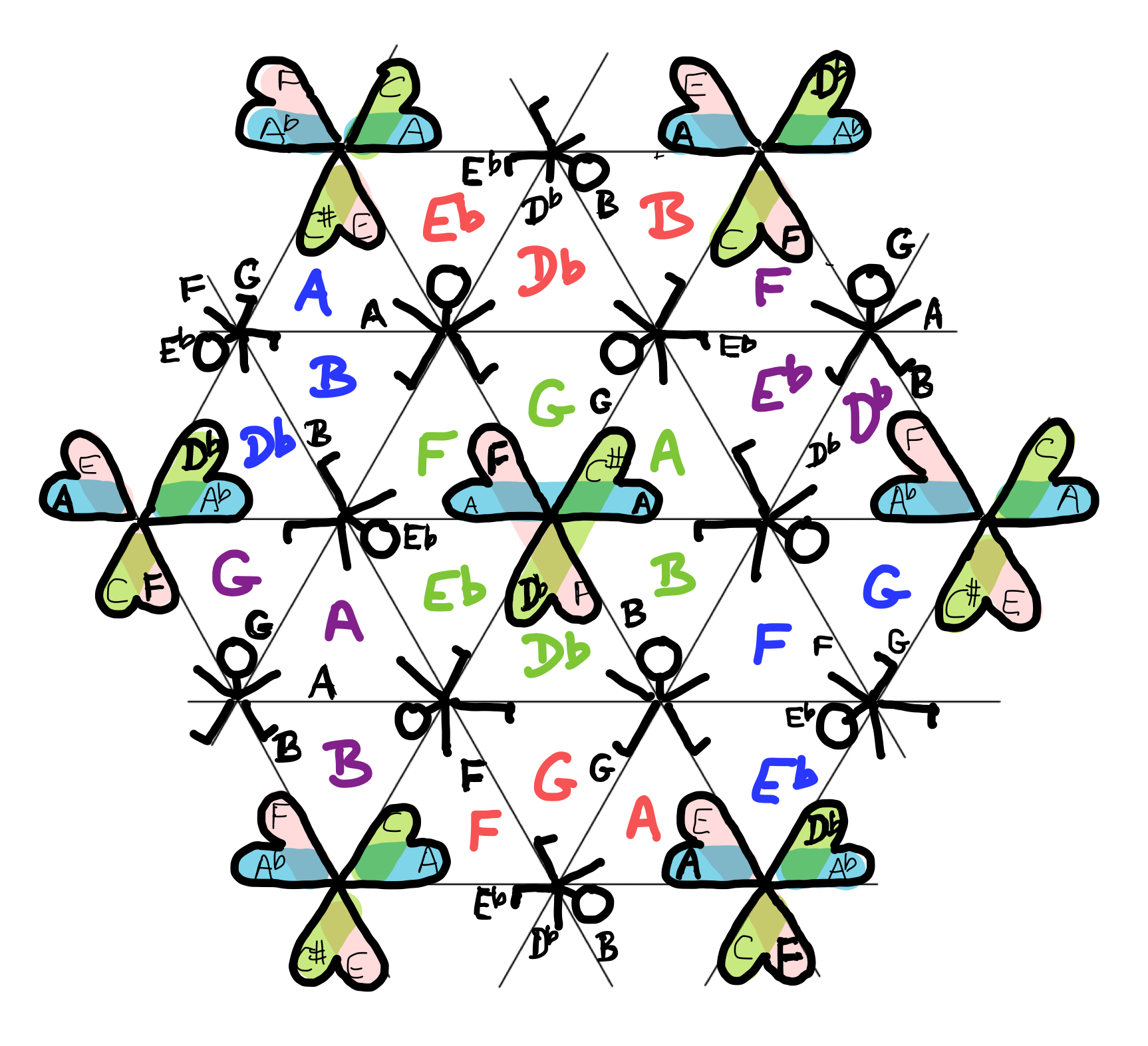}\\
\caption{A (stylised) depiction of the major tritone edge-tonnetz  from Figure~\ref{f:tritone2}}\label{f:tritone2hand}
\end{figure}

Finally, we observe that this tonnetz is again a periodic tiling of the plane, and opposite edges of the now enlarged fundamental hexagon can be identified to give a tonnetz on a torus. Note that this fundamental domain differs not only from the one in Figure~\ref{f:tritone1}, but also from the one of Euler's tonnetz, compare Figure~\ref{f:Euler}. This is despite the fact that both have fundamental domains consisting of $24$ triangles. The fundamental domain in Figure~\ref{f:tritone2} is remarkable in that it has a simply transitive on alcoves action of the group $S_4$, which relates it to the sphere in Section~\ref{s:sphere}.

\section{An edge-tonnetz on a sphere - the jazz bauble}\label{s:sphere}
All the constructions so far have been in the plane or on a torus. We now consider instead a sphere. There is already an example of a vertex tonnetz on the tetrahedron in the literature, see \cite[Figure 29]{Tymoczko12} and \cite[Figure~23]{Yust20}. This is given by labelling the vertices of the tetrahedron by $C, E\flat, F\sharp$ and $A$, equivalent to a diminished seventh chord. As a result each face represents a diminished triad, making this a diminished vertex-tonnetz in our terminology. 

Our final tonnetz is an edge-tonnetz constructed on a triangulation of a sphere with $24$ triangles. To describe the triangulation, map the sphere to a tetrahedron and now consider the four reflection hyperplanes that are each perpendicular to an edge. Refine the triangulation of the tetrahedron by adding in the intersection lines coming from intersecting with these hyperplanes. Each side of the tetrahedron gets divided up into $6$ smaller triangles, so that the four sides together give a triangulation with $24$ triangles. The idea to construct a tonnetz on the sphere with this triangulation was inspired by the observation found in \cite[Section~5.1]{FT-JLMS}, that the triangulated torus in Figure~\ref{f:tritone2} is a ramified double-cover of this sphere, with each triangle of the triangulated torus mapping to a compound triangle of the sphere (what are the equilateral triangles in Figure~\ref{f:sphere2}).

We describe this tonnetz in Figure~\ref{f:sphere1} using a net of the tetrahedron, where gray edges indicate where the folds should be. Our final tonnetz does not reproduce minor and major triads, like previous ones,  but it has another remarkable property. Consider the $12$ equilateral bold triangles in the net, which are the equilateral triangles that are folded in half when putting together the tetrahedron (and are really quadrilaterals from the point of view of our triangulation). Each one of these has precisely $5$ notes attached to it, and they are the notes of a major ninth chord. Moreover, each major ninth chord appears exactly once. For example, the edges involved in the bottom-most equilateral triangle in the lower right-hand corner are labelled by the notes of the major ninth chord based on $C$:
\[
C,E,G,B,D.
\] 
The arrangement of the major ninth chords can be understood as follows. The tetrahedron has four corners, labelled $X,Y,Z$ and $W$ in Figure~\ref{f:sphere2}, and around each one there are three major ninth chords. 
\begin{itemize}
\item The vertex $X$ has the chords based on $F,F\sharp$ and $G$ adjacent to it on its three faces.
\item The vertex $Y$ has the chords based on $A,B\flat$ and $B$ adjacent to it on its three faces.
\item The vertex $Z$ has the chords based on $D\flat,D$ and $E\flat$ adjacent to it on its three faces.
\item The final vertex, $W$, has the remaining chords, that are based on $C,E,A\flat$. 
\end{itemize}

The tonnetz has a symmetry of rotation around the axis through $W$. Namely, rotating by $120^\circ$ anticlockwise about the center of the net corresponds to transposition up by a major third.

\begin{figure}
\begin{center}
\begin{tikzpicture}[scale=0.7,
  xscale=3,yscale=5.199,
  note/.style={
  minimum size=0.75cm},
  every node/.append style={font=\footnotesize}
  ]




\draw [gray] (2.5,1)--(2.5,0);
\draw [gray](1,0.5)--(2.5,1);
\draw [gray](1,0.5)-- (2.5,0);

\draw[thick] (0.5,0)--(2,1.5);
\draw[thick] (0.5,0)--(3.5,0);
\draw[thick] (3.5,0)--(2,1.5);

\draw[thick] (1,0.5)--(4,0.5);
\draw[thick] (3,0.5)--(2.5,0)--(1,1.5);
\draw[thick] (2.5,1)--(1,-.5);

\draw[thick] (1,-0.5)--(0.5,0);  
\draw[gray] (1,-0.5)--(1,0.5);
\draw[thick] (1,0.5)--(1.5,0);
\draw[gray](2.5,0)--(4,0.5);
\draw[thick](4,0.5)--(3.5,0);
\draw[thick](1.5,1)--(2.5,1);
\draw[thick](1,1.5)--(2,1.5);
\draw[gray](2.5,1)--(1,1.5);

\draw (0.75,0) node[rectangle,fill=white!10]{ $D$}; 
\draw (1.25,0) node[rectangle,fill=white!10]{ $B$}; 
\draw (2,0) node[rectangle,fill=white!10]{ $B$}; 
\draw (3,0) node[rectangle,fill=white!10]{ $B$}; 

\draw (1.35,0.5) node[rectangle,fill=white!10]{ $E\sharp$}; 
\draw (1.65,0.5) node[rectangle,fill=white!10]{ $F$}; 
\draw (2.25,0.5) node[rectangle,fill=white!10]{ $F$}; 
\draw (2.75,0.5) node[rectangle,fill=white!10]{ $C$}; 
\draw (3.5,0.5) node[rectangle,fill=white!10]{ $G\sharp$}; 

\draw (2,1) node[rectangle,fill=white!10]{ $D$};

\draw (0.75,0.25) node[rectangle,fill=white!10]{ $G$}; 
\draw (1.25,0.25) node[rectangle,fill=white!10]{ $F\sharp$}; 
\draw (1.5,0.35) node[rectangle,fill=white!10]{ $A\sharp$}; 
\draw (2.1,0.15) node[rectangle,fill=white!10]{ $E$}; 
\draw (2.25,0.25) node[rectangle,fill=white!10]{ $A$}; 
\draw (2.68,0.2) node[rectangle,fill=white!10]{ $B\flat$};
\draw (2.85,0.35) node[rectangle,fill=white!10]{ $A\sharp$};
\draw (3.75,0.25) node[rectangle,fill=white!10]{ $E$}; 
\draw (1.65,0.15) node[rectangle,fill=white!10]{ $G\sharp$};
\draw (1.85,0.35) node[rectangle,fill=white!10]{ $C\sharp$};

\draw (1,0.2) node[rectangle,fill=white!10]{ $A$};

\draw (1,-0.2) node[rectangle,fill=white!10]{ $G$};

\draw (0.75,-0.25) node[rectangle,fill=white!10]{ $C$};
\draw (1.25,-0.25) node[rectangle,fill=white!10]{ $E$};

\draw (3,0.15) node[rectangle,fill=white!10]{ $C\sharp$};
\draw (3.5,0.33) node[rectangle,fill=white!10]{ $B$};

\draw (3.15,0.35) node[rectangle,fill=white!10]{ $D\sharp$};
\draw (3.35,0.15) node[rectangle,fill=white!10]{ $F\sharp$};

\draw (2.5,0.35) node[rectangle,fill=white!10]{ $D$};

\draw (1.25,1.25) node[rectangle,fill=white!10]{ $C$};
\draw (2.25,1.25) node[rectangle,fill=white!10]{ $E\flat$};

\draw (1.85,1.35) node[rectangle,fill=white!10]{ $B\flat$};

\draw (1.55,1.34) node[rectangle,fill=white!10]{ $E\flat$};

\draw (1.62,1.14) node[rectangle,fill=white!10]{ $G$};

\draw (2.1,1.14) node[rectangle,fill=white!10]{ $F$};

\draw (1.5,1.5) node[rectangle,fill=white!10]{ $A\flat$};

\draw (2,0.83) node[rectangle,fill=white!10]{ $F\sharp$};
\draw (1.25,0.75) node[rectangle,fill=white!10]{ $G$};
\draw (2.2,0.65) node[rectangle,fill=white!10]{ $C\sharp$};
\draw (2.3,0.8) node[rectangle,fill=white!10]{ $D\flat$};
\draw (2.75,0.75) node[rectangle,fill=white!10]{ $E\flat$};
\draw (2.5,0.65) node[rectangle,fill=white!10]{ $A\flat$};

\draw (1.5,0.66) node[rectangle,fill=white!10]{ $C$};
\draw (1.85,0.65) node[rectangle,fill=white!10]{ $A$};

\draw (1.6,0.85) node[rectangle,fill=white!10]{ $E$};







\end{tikzpicture}
\end{center}
\caption {Folding along the lighter-coloured edges and gluing the outer border according to (enharmonically equivalent) notes, gives an edge-tonnetz on a tetrahedron. 
}
\label{f:sphere1}
\end{figure}
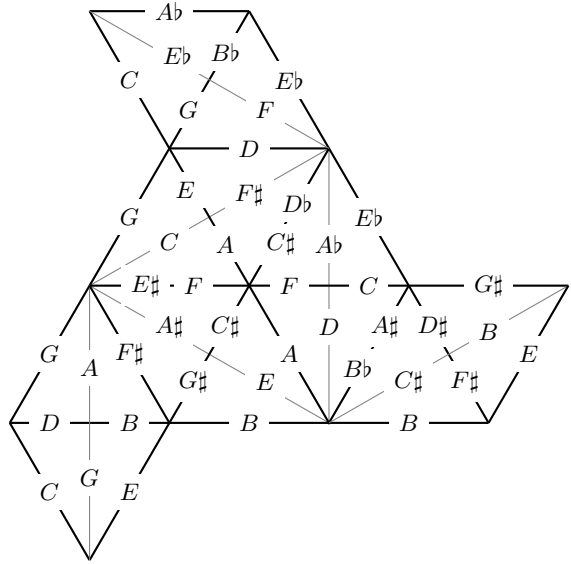
  
In Figure~\ref{f:sphere2} the pattern of chords is shown. It is interesting to note that, while the base note is always on one of the  long outer edges of its equilateral triangle, the remaining four notes of the ninth chord are arranged in different ways, as indicated by the different colours of the circles. Amongst these, the yellow chords, $D\flat, F$ and $A$ are arranged in the most harmonic way, with the corresponding major triad based along the inner sub-triangle and a stacked major triad starting on the $5$-th along the outer sub-triangle.

\begin{figure}
\begin{center}
\begin{tikzpicture}[scale=0.7,
  xscale=3,yscale=5.199,
  note/.style={
  minimum size=0.75cm},
  every node/.append style={font=\footnotesize}
  ]




\node (W) [left] at (1,1.5) {$W$};
\node (W) [left] at (1,-0.5) {$W$};
\node (W)[right] at (4,0.5) {$W$};
\node (X)[left] at (1,0.5) {$X$};
\node (Y)[below] at (2.5,0) {$Y$};
\node (Z)[right] at (2.5,1) {$Z$};

\draw [gray] (2.5,1)--(2.5,0);
\draw [gray](1,0.5)--(2.5,1);
\draw [gray](1,0.5)-- (2.5,0);

\draw[thick] (0.5,0)--(2,1.5);
\draw[thick] (0.5,0)--(3.5,0);
\draw[thick] (3.5,0)--(2,1.5);

\draw[thick] (1,0.5)--(4,0.5);
\draw[thick] (3,0.5)--(2.5,0)--(1,1.5);
\draw[thick] (2.5,1)--(1,-.5);

\draw[thick] (1,-0.5)--(0.5,0);  
\draw[gray] (1,-0.5)--(1,0.5);
\draw[thick] (1,0.5)--(1.5,0);
\draw[gray](2.5,0)--(4,0.5);
\draw[thick](4,0.5)--(3.5,0);
\draw[thick](1.5,1)--(2.5,1);
\draw[thick](1,1.5)--(2,1.5);
\draw[gray](2.5,1)--(1,1.5);

\draw (1,-0.2) node[draw,circle,fill=lightgray!20]{ $C$}; 
\draw (1,0.17) node[draw,circle,fill=white!10]{ $G$};
\draw (2,0.16) node[draw,circle,fill=yellow!20]{ $A$}; 
\draw (3,0.17) node[draw,circle,fill=white!10]{ $B$}; 
\draw (2.5,0.33) node[draw,circle,fill=darkgray!20]{ $B\flat$}; 
\draw (3.5,0.33) node[draw,circle,fill=lightgray!20]{ $E$}; 
\draw (1.5,0.33) node[draw,circle,fill=darkgray!20]{ $F\sharp$}; 
\draw (2,0.81) node[draw,circle,fill=darkgray!20]{ $D$}; 
\draw (2.5,0.67) node[draw,circle,fill=yellow!20]{ $D\flat$}; 
\draw (1.5,0.66) node[draw,circle,fill=yellow!20]{ $F$}; 
\draw (2,1.16) node[draw,circle,fill=white!10]{ $E\flat$}; 
\draw (1.5,1.33) node[draw,circle,fill=lightgray!20]{ $A\flat$}; 

\end{tikzpicture}
\end{center}
\caption {In Figure~\ref{f:sphere1}, two triangles together form a major $9$th chord with the base note shown here.}\label{f:sphere2}
\end{figure}
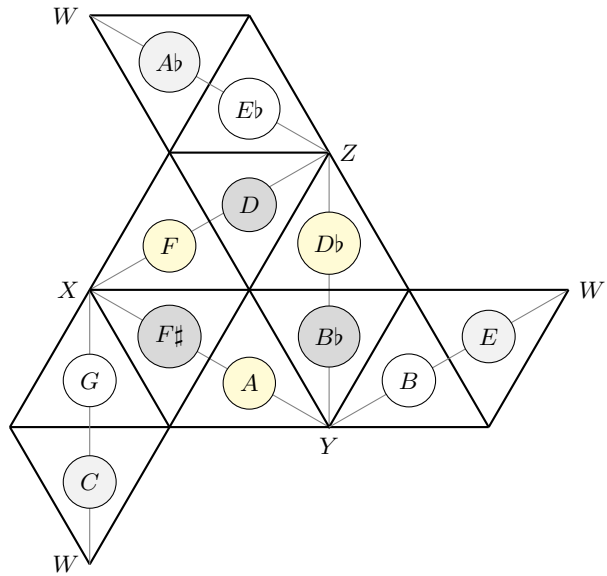
    
  \clearpage
\section{Overview of chords appearing in our type $B_2$ and $G_2$ tonnetz examples}\label{s:table}
The following table shows the chords that are found in the examples from Sections~\ref{s:BC} and \ref{s:G2}.~\begin{table}[h]
  \begin{center}
    \label{tab:table1}
    \begin{tabular}{l|c|c|c|c}
      \textbf{} & \textbf{\dynkin[edge length=.8cm,
root radius=.09cm,labels={A,F}]{B}{2}}&\textbf{\dynkin[edge length=.8cm,
root radius=.09cm,labels={A,F}]{C}{2}}& \textbf{\dynkin[edge length=.8cm,
root radius=.09cm,labels={A,D}]{G}{2}} & \textbf{\dynkin[edge length=.8cm,
root radius=.09cm,labels={D,A}]{G}{2}}\\
\hline\hline
    \textbf{major triads}&&&&\\
    \hline
\textit{faces:}  &$D$-major& & $D$-major & $A$-major \\
 &$B$-major&& $E$-major & $B$-major \\
 && & $F\sharp$-major & $D\flat$-major \\
 &$A\flat$-major&& $A\flat$-major & $E\flat$-major \\
 &$F$-major&& $B\flat$-major & $F$-major \\
 &&& $C$-major & $G$-major \\
    \hline \hline
   \textbf{minor triads} &&&&\\
    \hline
\textit{faces:}  &&$D$-minor & $A$-minor & $D$-minor \\
&& $B$-minor& $B$-minor & $E$-minor \\
 && & $C\sharp$-minor & $F\sharp$-minor \\
 &&$G\sharp$-minor& $E\flat$-minor & $G\sharp$-minor \\
 &&$F$-minor& $F$-minor & $B\flat$-minor \\
 &&& $G$-minor & $C$-minor \\
     \hline
      \hline
      \textit{all edges: } & $\mathcal N\setminus\{A\sharp, C\sharp,E,G\}$&$\mathcal N\setminus\{A\sharp, C\sharp,E,G\}$& $\mathcal N$ & $\mathcal N$ \\
      \hline\hline
     \textbf{augmented triads} &&&  &  \\
      \hline
   \textit{ valence $6$ vertices:}&& & $C,E,G\sharp$ & $F,A,C\sharp$\\
    && & $D, F\sharp, A\sharp$ & $G,B,D\sharp$\\
        \hline
         \hline
   \textbf{diminished seventh}&&&&\\     
            \hline
   \textit{ valence $4$ vertices:} &$A,C,E\flat,G\flat$&& $A,C,E\flat,G\flat$& $A,C,E\flat,G\flat$\\
   && & $A\sharp,C\sharp,E,G$ & $A\sharp,C\sharp,E,G$ \\
       &&  $B,D,F,A\flat$& $B,D,F,A\flat$  & $B,D,F,A\flat$\\
           \hline
      \hline
         \textbf{anything else}&&&&\\     
            \hline
 \textit{valence $8$ vertices:}&$\mathcal N\setminus\{A\sharp, C\sharp,E,G\}$ & $\mathcal N\setminus\{A\sharp, C\sharp,E,G\}$& & \\
 \textit{valence $12$ vertices:}& && &  $\mathcal N$\\
\hline      \hline
      \end{tabular}
  \end{center}
\end{table}

\newpage 
\bibliographystyle{amsalpha}

\end{document}